%% file: cylinder_log-CY_surf_I.tex
\pdfoutput=1
\newif\ifpersonal
\documentclass[12pt,a4paper,reqno]{amsart} 
\usepackage{amsmath,amsthm,amssymb,mathrsfs,mathtools,bm} 
\usepackage{microtype,fixltx2e,lmodern,textcomp} 
\usepackage{enumerate,comment,braket,xspace,tikz-cd} 
\usepackage[utf8]{inputenc} 
\usepackage[T1]{fontenc} 
\usepackage[centering,vscale=0.7,hscale=0.7]{geometry}
\usepackage[hidelinks]{hyperref}
\usepackage[capitalize]{cleveref}

\theoremstyle{plain}
\newtheorem{thm}{Theorem}[section]
\newtheorem{lem}[thm]{Lemma}
\newtheorem{prop}[thm]{Proposition}
\newtheorem{cor}[thm]{Corollary}
\theoremstyle{definition}
\newtheorem{defin}[thm]{Definition}

\theoremstyle{remark}
\newtheorem{eg}[thm]{Example}
\newtheorem{rem}[thm]{Remark}
\numberwithin{equation}{section}

\input{newcommands_all}

\begin{document}
\title[Enumeration of holomorphic cylinders. I]{Enumeration of holomorphic cylinders\\ in log Calabi-Yau surfaces. I}
\author{Tony Yue YU}
\address{Tony Yue YU, Institut de Mathématiques de Jussieu - Paris Rive Gauche, CNRS-UMR 7586, Case 7012, Université Paris Diderot - Paris 7, Bâtiment Sophie Germain 75205 Paris Cedex 13 France}
\email{yuyuetony@gmail.com}
\date{April 7, 2015 (revised on August 23, 2016)}
\subjclass[2010]{Primary 14N35; Secondary 14J32 14J26 14T05 14G22}
\keywords{cylinder, broken line, wall-crossing, enumerative geometry, non-archimedean geometry, Berkovich space, log Calabi-Yau, del Pezzo surface}

\begin{abstract}
We define the counting of holomorphic cylinders in log Calabi-Yau surfaces.
Although we start with a complex log Calabi-Yau surface, the counting is achieved by applying methods from non-archimedean geometry.
This gives rise to new geometric invariants.
Moreover, we prove that the counting satisfies a property of symmetry.
Explicit calculations are given for a del Pezzo surface in detail,
which verify the conjectured wall-crossing formula for the focus-focus singularity.
Our holomorphic cylinders are expected to give a geometric understanding of the combinatorial notion of broken line by Gross, Hacking, Keel and Siebert.
Our tools include Berkovich spaces, tropical geometry, Gromov-Witten theory and the GAGA theorem for non-archimedean analytic stacks.
\end{abstract}

\maketitle

\personal{Personal comments are shown!}

\tableofcontents

\input{cylinder_log-CY_surf_I_body}

\bibliographystyle{plain}
\bibliography{dahema}

\end{document}

%% file: newcommands_all.tex
\ifpersonal
\newcommand*{\personal}[1]{\textcolor[rgb]{0,0,1}{(Personal: #1)}}
\newcommand*{\todo}[1]{\textcolor{red}{(Todo: #1)}}
\else
\newcommand*{\personal}[1]{\ignorespaces}
\newcommand*{\todo}[1]{\ignorespaces}
\fi

\newcommand{\C}{\mathbb C}

\newcommand{\R}{\mathbb R}
\newcommand{\Z}{\mathbb Z}

\newcommand{\fC}{\mathfrak C}

\newcommand{\fX}{\mathfrak X}

\newcommand{\ff}{\mathfrak f}

\newcommand{\cF}{\mathcal F}

\newcommand{\cM}{\mathcal M}

\newcommand{\cO}{\mathcal O}

\newcommand{\cX}{\mathcal X}

\DeclareFontFamily{U}{BOONDOX-calo}{\skewchar\font=45 }
\DeclareFontShape{U}{BOONDOX-calo}{m}{n}{<-> s*[1.05] BOONDOX-r-calo}{}
\DeclareFontShape{U}{BOONDOX-calo}{b}{n}{<-> s*[1.05] BOONDOX-b-calo}{}
\DeclareMathAlphabet{\mathcalboondox}{U}{BOONDOX-calo}{m}{n}

\newcommand{\bbG}{\mathbb G}

\newcommand{\bbP}{\mathbb P}

\newcommand{\be}{\mathbf e}

\newcommand{\sX}{\mathscr X}

\newcommand{\sU}{\mathscr U}

\makeatletter
\let\save@mathaccent\mathaccent
\newcommand*\if@single[3]{%
	\setbox0\hbox{${\mathaccent"0362{#1}}^H$}%
	\setbox2\hbox{${\mathaccent"0362{\kern0pt#1}}^H$}%
	\ifdim\ht0=\ht2 #3\else #2\fi
}
\newcommand*\rel@kern[1]{\kern#1\dimexpr\macc@kerna}
\newcommand*\widebar[1]{\@ifnextchar^{{\wide@bar{#1}{0}}}{\wide@bar{#1}{1}}}
\newcommand*\wide@bar[2]{\if@single{#1}{\wide@bar@{#1}{#2}{1}}{\wide@bar@{#1}{#2}{2}}}
\newcommand*\wide@bar@[3]{%
	\begingroup
	\def\mathaccent##1##2{%
		\let\mathaccent\save@mathaccent
		\if#32 \let\macc@nucleus\first@char \fi
		\setbox\z@\hbox{$\macc@style{\macc@nucleus}_{}$}%
		\setbox\tw@\hbox{$\macc@style{\macc@nucleus}{}_{}$}%
		\dimen@\wd\tw@
		\advance\dimen@-\wd\z@
		\divide\dimen@ 3
		\@tempdima\wd\tw@
		\advance\@tempdima-\scriptspace
		\divide\@tempdima 10
		\advance\dimen@-\@tempdima
		\ifdim\dimen@>\z@ \dimen@0pt\fi
		\rel@kern{0.6}\kern-\dimen@
		\if#31
		\overline{\rel@kern{-0.6}\kern\dimen@\macc@nucleus\rel@kern{0.4}\kern\dimen@}%
		\advance\dimen@0.4\dimexpr\macc@kerna
		\let\final@kern#2%
		\ifdim\dimen@<\z@ \let\final@kern1\fi
		\if\final@kern1 \kern-\dimen@\fi
		\else
		\overline{\rel@kern{-0.6}\kern\dimen@#1}%
		\fi
	}%
	\macc@depth\@ne
	\let\math@bgroup\@empty \let\math@egroup\macc@set@skewchar
	\mathsurround\z@ \frozen@everymath{\mathgroup\macc@group\relax}%
	\macc@set@skewchar\relax
	\let\mathaccentV\macc@nested@a
	\if#31
	\macc@nested@a\relax111{#1}%
	\else
	\def\gobble@till@marker##1\endmarker{}%
	\futurelet\first@char\gobble@till@marker#1\endmarker
	\ifcat\noexpand\first@char A\else
	\def\first@char{}%
	\fi
	\macc@nested@a\relax111{\first@char}%
	\fi
	\endgroup
}
\makeatother

\newcommand{\oD}{\widebar D}

\newcommand{\oU}{\widebar U}

\newcommand{\oY}{\widebar Y}

\newcommand{\ov}{\widebar v}
\newcommand{\ox}{\widebar x}
\newcommand{\oy}{\widebar y}

\newcommand{\hh}{\widehat h}
\newcommand{\hC}{\widehat C}
\newcommand{\hE}{\widehat E}
\newcommand{\hF}{\widehat F}

\newcommand{\hL}{\widehat L}
\newcommand{\hU}{\widehat U}

\newcommand{\hbeta}{\widehat\beta}
\newcommand{\hGamma}{\widehat\Gamma}

\newcommand{\tW}{\widetilde W}
\newcommand{\tk}{\tilde k}

\newcommand{\tB}{\widetilde B}
\newcommand{\tD}{\widetilde D}

\newcommand{\tX}{\widetilde X}
\newcommand{\tfX}{\widetilde{\fX}}

\newcommand{\tsX}{\widetilde{\sX}}
\newcommand{\tH}{\widetilde H}
\newcommand{\tY}{\widetilde Y}
\newcommand{\tbeta}{\widetilde{\beta}}

\newcommand{\ttau}{\widetilde{\tau}}

\newcommand{\Ih}{I^\mathrm{h}}
\newcommand{\Iv}{I^\mathrm{v}}

\newcommand{\SsXH}{S_{(\sX,H)}}

\newcommand{\CsXH}{C_{(\sX,H)}}

\newcommand{\sXe}{\sX_\eta}
\newcommand{\sXs}{\sX_s}

\newcommand{\bcM}{\widebar{\mathcal M}}

\newcommand{\bcMgn}{\widebar{\mathcal M}_{g,n}}

\newcommand{\bcMon}{\widebar{\mathcal M}_{0,n}}

\newcommand{\Gmk}{\mathbb G_{\mathrm m/k}}

\newcommand{\Gmknan}{(\Gmk^n)\an}

\newcommand{\kc}{k^\circ}
\newcommand{\llb}{[\![}
\newcommand{\rrb}{]\!]}
\newcommand{\llp}{(\!(}
\newcommand{\rrp}{)\!)}
\newcommand{\an}{^\mathrm{an}}
\newcommand{\alg}{^\mathrm{alg}}

\newcommand{\ev}{\mathit{ev}}

\newcommand{\inv}{^{-1}}
\newcommand{\id}{\mathrm{id}}

\newcommand{\kanal}{$k$-analytic\xspace}

\newcommand{\red}{^\mathrm{red}}
\renewcommand{\th}{^\mathrm{\tiny th}}

\newcommand{\Zaffine}{$\mathbb Z$-affine\xspace}
\newcommand{\sw}{^\mathrm{sw}}

\usetikzlibrary{decorations.markings} 
\tikzset{
  closed/.style = {decoration = {markings, mark = at position 0.5 with { \node[transform shape, xscale = .8, yscale=.4] {/}; } }, postaction = {decorate} },
  open/.style = {decoration = {markings, mark = at position 0.5 with { \node[transform shape, scale = .7] {$\circ$}; } }, postaction = {decorate} }
}

\DeclareMathOperator{\Int}{Int}

\DeclareMathOperator{\NE}{NE}
\DeclareMathOperator{\oStar}{\widebar{\Star}}

\DeclareMathOperator{\Spec}{Spec}
\DeclareMathOperator{\Spf}{Spf}

\DeclareMathOperator{\Star}{Star}

\DeclareMathOperator{\val}{val}

%% file: cylinder_log-CY_surf_I_body.tex
\section{Introduction}\label{sec:intro_log-CY}

In mirror symmetry, the enumeration of holomorphic discs is of great importance.
Holomorphic discs are the building blocks of the Fukaya category \cite{Fukaya_Morse_homotopy_1993,Kontsevich_Homological_algebra_1995,Fukaya_Lagrangian_intersection_2009}.
Holomorphic discs also play the role of ``quantum corrections'' in the reconstruction of mirror manifolds \cite{Fukaya_Multivalued_Morse_2005,Kontsevich_Affine_2006,Auroux_Mirror_symmetry_anticanonical_2007,Fukaya_Cyclic_symmetry_2010,Tu_On_the_reconstruction_2014}.
More precisely, the ``quantum corrections'' arise from counting holomorphic discs with boundaries on torus fibers of an SYZ fibration (Strominger-Yau-Zaslow \cite{SYZ_1996}).

It turns out to be insufficient to restrict to holomorphic discs.
One can enrich the geometry by counting not only discs but also more general Riemann surfaces with boundaries on torus fibers of an SYZ fibration.
In this paper, we study a special case: the counting of \emph{holomorphic cylinders} in log Calabi-Yau surfaces.
The general case would require much more foundational efforts.

Our considerations are very much motivated by the work of Gross-Hacking-Keel \cite{Gross_Mirror_Log_2011}\footnote{In this paper, we will always refer to the first arXiv version \cite{Gross_Mirror_Log_2011}, because it contains much more material than the published version \cite{Gross_Mirror_Log_published}.}.
A special case of our holomorphic cylinders is related to the remarkable notion of broken line in loc.\ cit.
Broken lines are combinatorial objects which are responsible for the construction of the Landau-Ginzburg potential and the theta functions on the mirror manifold.
It is developed by Gross, Hacking, Keel, Siebert and their coauthors in a series of papers \cite{Gross_Mirror_symmetry_for_P2_2010,Carl_A_tropical_view_2010,Gross_Mirror_Log_2011,Gross_Theta_2012,Gross_Theta_2015} (also suggested by Abouzaid, Kontsevich and Soibelman).
Our holomorphic cylinders, besides their own interest, are expected to give a geometric interpretation of the broken lines, as well as a better understanding of the canonical theta functions in mirror symmetry.

Different from the existing literatures, we will work in the framework of non-archimedean analytic geometry à la Berkovich \cite{Berkovich_Spectral_1990,Berkovich_Etale_1993} instead of differential geometry.
The reason is that it is easier to apply techniques from tropical geometry in the non-archimedean setting.
It is helpful to think of the non-archimedean picture as the most degenerate differential-geometric picture.
The total degeneracy makes many constructions and proofs more transparent.
The non-archimedean approach to mirror symmetry was first suggested by Kontsevich and Soibelman in \cite{Kontsevich_Homological_2001}, where it is expected that the differential-geometric picture and the non-archimedean picture should be equivalent.

\bigskip
Now let us explain our paper in precise mathematical terms.

We start with a Looijenga pair\footnote{The terminology is borrowed from \cite{Gross_Mirror_Log_2011}. The notion was extensively studied by Looijenga \cite{Looijenga_Rational_1981}.} $(Y,D)$, i.e.\ a connected smooth complex projective surface $Y$ together with a singular nodal curve $D$ representing the anti-canonical class $-K_Y$.
Let $k\coloneqq\C\llp t\rrp$ be the field of formal Laurent series.
Let $X\coloneqq Y\setminus D$.
Let $X_k\coloneqq X\times_{\Spec\C}\Spec k$ be the base change and $X\an_k$ the non-archimedean analytification.
Using a variant of Berkovich's deformation retraction \cite{Berkovich_Smooth_1999}, we construct a proper continuous map $\tau\colon X\an_k\to B$, where $B$ is a topological space homeomorphic to $\R^2$.
The map $\tau$ is a non-archimedean avatar of the SYZ fibration in differential geometry (see \cref{prop:extension_of_torus_fibration}).
The generic fibers of $\tau$ are non-archimedean affinoid tori.

Our goal in this paper is to define the counting of holomorphic cylinders in $X\an_k$ with boundaries on two different torus fibers of $\tau\colon X\an_k\to B$.
Ideally, we would like to consider the moduli space of such holomorphic cylinders, and then define the counting using this moduli space.
However, in general, a holomorphic cylinder in $X\an_k$ can be very wild and complicated.
Therefore, we impose a regularity condition on the boundaries of the cylinders: we require that when we make analytic continuation at the boundaries, our cylinders extend straight to infinity.
This regularity condition helps us relate the counting of cylinders to the counting of certain closed curves in $Y\an_k$, i.e. certain Gromov-Witten type invariants.
Finally, we achieve the counting via a mixture of Gromov-Witten theory, non-archimedean geometry and tropical geometry.

The counting of holomorphic cylinders is expected to satisfy a list of nice properties.
We prove in this paper one non-trivial fundamental property called the \emph{symmetry property} (cf.\ \cref{thm:symmetry}).
This ensures that our non-archimedean construction is compatible with the intuition from symplectic geometry.
Here is a heuristic explanation of the symmetry property:
Let $F_1$ and $F_2$ be two different torus fibers of the map $\tau\colon X\an_k\to B$.
The actual construction of our counting invariants depends on the orientation of the cylinders.
In other words, we have the number of holomorphic cylinders going from $F_1$ to $F_2$, and the number of holomorphic cylinders going from $F_2$ to $F_1$.
The symmetry property states that the two numbers are equal.
We will explore other properties of our counting invariants besides the symmetry property in subsequent works.

Since the counting of holomorphic cylinders is constructed in a rather indirect way,
we give a concrete computation for a del Pezzo surface in the end of this paper.
We prove that the corresponding numbers of cylinders are certain binomial coefficients.
Our computation verifies the Kontsevich-Soibelman wall-crossing formula for the focus-focus singularity (cf.\ \cite{Kontsevich_Affine_2006}).

\bigskip

Here is the outline of this paper:

In \cref{sec:review_tropicalizations}, we start with a review of tropicalization, toroidal modification and integral affine structure.
In \cref{sec:log-CY_surfaces}, we study the non-archimedean SYZ fibration of a log Calabi-Yau surface.
In \cref{sec:tropical_cylinders}, we introduce several combinatorial constructions on the base $B$: spines, tropical cylinders, extension of spines and extension of tropical cylinders.
Moreover, we prove a rigidity property of tropical cylinders (\cref{prop:rigidity_of_tropical_cylinder}), which will be an important ingredient in the proof of the properness of the moduli space in \cref{sec:hol_cylinders}.

In \cref{sec:hol_cylinders}, we define our counting invariant: the \emph{number $N(L,\beta)$ of homomorphic cylinders} in $X\an_k$, given a spine $L$ in $B$ and a curve class $\beta$ in $Y$.
By imposing the regularity condition as explained above, we relate the number of cylinders to certain Gromov-Witten type invariants.
We use non-archimedean geometry here to cut out relevant components of the moduli space of stable maps.
Finally, thanks to the GAGA theorem for non-archimedean analytic stacks proved in \cite{Porta_Yu_Higher_analytic_stacks_2014},
we are able to resort to the virtual fundamental classes in algebraic geometry and achieve the enumeration.

In \cref{sec:symmetry}, we prove the symmetry property of the counting of holomorphic cylinders.
In \cref{sec:example_del_Pezzo}, we carry out explicit computations of our counting invariants for a del Pezzo surface in detail.

\bigskip

\paragraph{\bfseries Related works}
Our previous works on non-archimedean geometry and tropical geometry \cite{Yu_Balancing_2013,Yu_Gromov_2014,Yu_Number_2013,Yu_Tropicalization_2014,Porta_Yu_Higher_analytic_stacks_2014} provide general foundations for the context of this paper.
We will refer to \cite[\S 6]{Yu_Gromov_2014} and \cite{Porta_Yu_Higher_analytic_stacks_2014} for the theory of stacks in non-archimedean analytic geometry.

As mentioned in the beginning, we are very much inspired by the notion of broken line in mirror symmetry.
The collection of broken lines is believed to be a more fundamental object than the scattering diagram (cf.\ \cite[Remark 0.21]{Gross_Mirror_Log_2011}).
We refer to \cite{Kontsevich_Affine_2006,Gross_Tropical_vertex_2010,Gross_Real_Affine_2011} for the notion of scattering diagram.
However, the definition of broken line in \cite{Gross_Mirror_Log_2011} is based on the scattering diagram and is combinatoric in nature.
It is expected that our consideration on holomorphic cylinders leads to a precise geometric understanding of the broken lines, which does not à priori refer to scattering diagrams.
We will explore this direction in subsequent works.

We would also like to mention that the works by Auroux \cite{Auroux_Mirror_symmetry_anticanonical_2007,Auroux_Special_Lagrangian_fibrations_2009}, Nishinou \cite{Nishinou_Disk_2012} and Lin \cite{Lin_Open_GW_2014} on the enumeration of holomorphic discs are in a similar spirit of this paper.

\medskip

\paragraph{\textbf{Acknowledgments}}
I am very grateful to Maxim Kontsevich for suggesting this direction of research and sharing with me many fruitful ideas.
Special thanks to Antoine Chambert-Loir for continuous support.
During revision, Sean Keel suggested me a better way to deal with the curve classes.
I am equally grateful to Luis Alvarez-Consul, Denis Auroux, Vladimir Berkovich, Benoît Bertrand, Philip Boalch, Olivier Debarre, Lie Fu, Mark Gross, Ilia Itenberg, Mattias Jonsson, François Loeser, Ernesto Lupercio, Grigory Mikhalkin, Johannes Nicaise, Johannes Rau, Yan Soibelman, Jake Solomon, Michael Temkin and Bertrand Toën for their helpful comments, and for providing me opportunities to present this work in various seminars and conferences.

\section{Tropicalization and integral affine structures} \label{sec:review_tropicalizations}

First, we describe a general setup of tropicalization using snc pairs (see also \cite{Berkovich_Smooth_1999,Kontsevich_Non-archimedean_2002,Thuillier_Geometrie_toroidale_2007,Boucksom_Singular_2011,Gubler_Skeletons_2014,Yu_Tropicalization_2014}).
Then we recall the relation between toroidal blowups of formal models and polyhedral subdivisions of the tropicalization following \cite{Kempf_Toroidal_1973}.
Finally we review integral affine structures.

Let $k=\C \llp t \rrp$ be the field of formal Laurent series.
It has the structure of a complete discrete valuation field with $t$ being a uniformizer.
Let $k^\circ=\C\llb t \rrb$ denote the ring of integers, $\widetilde k=\C$ the residue field, and $\val\colon k^\times \to \Z$ the valuation map.


For a $\kc$-scheme $\sX$, we denote by $\sXe$ its generic fiber over $k$, and by $\sXs$ its special fiber over the residue field $\tk$.
For a scheme $X$ locally of finite type over $k$, we denote by $X\an$ the analytification of $X$ (cf.\ \cite{Berkovich_Spectral_1990}).

\begin{defin}\label{def:snc_pair}
An \emph{snc pair} $(\sX,H)$ consists of a proper flat $\kc$-scheme and a finite sum $H=\sum_{i\in\Ih} D_i$ of distinguished effective Cartier divisors on $\sX$ such that
\begin{enumerate}[(i)]
\item every $D_i$ has irreducible support,
\item every point of $\sX$ has an open affine neighborhood $\sU$ which admits an étale morphism
\begin{equation}\label{eq:local_coordinates}
\phi\colon\sU\longrightarrow \Spec \left(k^\circ[T_0,\dots,T_n]/(T_0^{m_0}\cdots T_{d}^{m_d}-\varpi)\right)
\end{equation}
for some $0\le d\le n$, $m_0,\dots,m_d\in\Z_{>0}$ and a uniformizer $\varpi$ of $k$,
and that for every $i\in \Ih$ the restriction $D_i|_{\sU}$ is either empty or defined by $\phi^*(T_j)$ for some $j>d$.
\end{enumerate}
\end{defin}

\begin{rem}
\cref{def:snc_pair} is a variant of \cite[Definition 3.1]{Gubler_Skeletons_2014} by allowing multiplicities.
We remark that the notion of formal strictly semistable pair in Definition 3.7 loc.\ cit.\ can also be generalized to the notion of \emph{formal snc pair} in the same way.
Nevertheless, (algebraic) snc pairs are more convenient for this paper.
\end{rem}

Let $(\sX,H)$ be an snc pair.
We denote by $\Set{D_i}_{i\in\Iv}$ the finite set of irreducible components of the special fiber $\sXs\red$ with its reduced scheme structure.
For every $i\in\Iv$, let $\mathit{mult}_i$ denote the multiplicity of $D_i$ in $\sXs$.
The divisors in the set $\Ih$ are called \emph{horizontal divisors}, while the divisors in the set $\Iv$ are called \emph{vertical divisors}.
For every non-empty subset $I\subset \Iv\cup \Ih$, put $D_I\coloneqq\cap_{i\in I} D_i$,
$D_I^\circ\coloneqq D_I\setminus\big(\bigcup_{i\in (\Iv\cup\Ih)\setminus I} D_i\big)$.
We further assume that every $D_I$ is either empty or irreducible.

\begin{defin}[cf.\ \cite{Kempf_Toroidal_1973,Kontsevich_Non-archimedean_2002}]
We define the \emph{Clemens cone} and the \emph{Clemens polytope} associated to an snc pair $(\sX,H)$ to be respectively
\begin{align*}
\CsXH &= \bigg\{\sum_{i\in \Iv\cup\Ih} a_i \langle D_i \rangle \ \bigg|\  a_i\ge 0, \bigcap_{i\,:\, a_i>0} D_i \neq \emptyset \bigg\} \subset \R^{\Iv\cup\Ih},\\
\SsXH &= \CsXH \cap \Set{\sum_{i\in \Iv} \mathit{mult}_i\cdot a_i=1}.
\end{align*}
\end{defin}

\begin{defin}\label{def:snc_log_model}
Let $\cX$ be an algebraic variety over $k$.
An \emph{snc log-model} of $\cX$ consists of an snc pair $(\sX,H)$ together with an isomorphism $X\simeq(\sX\setminus H)_\eta$.
\end{defin}

As in \cite{Berkovich_Smooth_1999,Thuillier_Geometrie_toroidale_2007,Nicaise_Essential_skeleton_2013,Gubler_Skeletons_2014}, one has a canonical embedding $\SsXH\hookrightarrow X\an$ and a canonical strong deformation retraction from $X\an$ to $\SsXH$.
In this paper, we will only be concerned with the retraction map at time one, which we denote by $\tau\colon X\an\to\SsXH$.
It has a simple description below.

Locally on the formal completion of $\sX$ along its special fiber, a divisor $D_i$ for $i\in \Iv\cup\Ih$ is given by a function $u_i$,
which is well-defined up to multiplication by invertible functions.
So $\val(u_i(x))$ defines a continuous function on $X\an$.
Then the map $\tau$ equals the following continuous map
\[X\an\longrightarrow\R^{\Iv\cup\Ih},\qquad x\longmapsto \big(\val (u_i(x))\big)_{i\in \Iv\cup\Ih},\]
whose image coincides with the Clemens polytope $\SsXH$.

\subsection*{Toroidal modifications}

We restrict to snc log-models for simplicity rather than for necessity.
A more general framework is to use toroidal log-models, in the sense that the pair $\big(\sX,\sXs\cup H\big)$ is étale locally isomorphic to a toric scheme over $\kc$ with its toric boundary (cf.\ \cite[\S 4.3]{Kempf_Toroidal_1973}, \cite[\S 7]{Gubler_Guide_2013}).
In the toroidal case, the Clemens cone $\CsXH$ would be the conical polyhedral complex defined in Chapter II \S 1 loc.\ cit.,
and the Clemens polytope would be an analog of the compact polyhedral complex in the end of Chapter II \S 3 loc.\ cit.

Given a finite rational polyhedral subdivision $\CsXH'$  of the Clemens cone $\CsXH$, Chapter II Theorems 6* and 8* in \cite{Kempf_Toroidal_1973} allow us to construct a new toroidal log-model $(\sX',H')$ of $X$ dominating the original one, such that
\[C_{(\sX',H')} \simeq C'_{(\sX,H)},\qquad S_{(\sX',H')} \simeq S'_{(\sX,H)},\]
where $\SsXH'$ denotes the subdivision of $\SsXH$ induced by the subdivision $\CsXH'$.

As we are mainly interested in snc log-models, we remark that the new log-model $(\sX',H')$ is an snc pair if and only if $\CsXH'$ is a finite rational subdivision of $\CsXH$ into simplicial cones whose integer points can be generated by a subset of a basis of the lattice $\Z^{\Iv\cup\Ih}$ (cf.\ \cite[Chapter II Theorem 4*]{Kempf_Toroidal_1973}).
We call such subdivisions \emph{regular simplicial subdivisions}.

\subsection*{Integral affine structures}

\begin{defin}\label{def:integral_affine_structure}
An $n$-dimensional chart on a paracompact Hausdorff topological space $B$ is a pair $(U,\phi)$,
where $U$ is a open subset of $B$, and $\phi\colon U\to \R^n$ is a homeomorphism of $U$ onto an open subset of a convex polyhedron in $\R^n$ not contained in a hyperplane.
An ($n$-dimensional) \emph{integral affine structure} (\emph{\Zaffine structure} for short) on $B$ is a maximal collection of $n$-dimensional charts such that the transitions functions belong to the group $\mathrm{GL}(n, \Z) \ltimes \R^n$ of integral affine transformations of $\R^n$.
\end{defin}

\begin{rem}
	When $n=1$, choosing a \Zaffine structure is the same as choosing a metric.
\end{rem}

\cref{def:integral_affine_structure} is an extension of \cite[\S 2.1]{Kontsevich_Affine_2006} to manifolds with corners.

A real-valued function $f$ on $\R^n$ is said to be \emph{integral affine} if it has the form
\[f(x_1,\dots,x_n)=a_1 x_1+\dots+a_n x_n+b\]
where $a_1,\dots,a_n\in\Z$ and $b\in\R$.
Let $\Delta$ be a convex polyhedron in $\R^n$ not contained in a hyperplane.
We denote by $\mathrm{Aff}_{\Z,\Delta}$ the sheaf of functions on $\Delta$ which are locally integral affine.

Since a homeomorphism between two open domains in $\R^n$ preserving the sheaf $\mathrm{Aff}_{\Z,\R^n}$ is necessarily given by an integral affine transformation,
an integral affine structure on a paracompact Hausdorff topological space $B$ can be given equivalently as a subsheaf $\mathrm{Aff}_{\Z,B}$ of the sheaf of continuous functions on $B$ such that the pair $(B,\mathrm{Aff}_{\Z,B})$ is locally isomorphic to $(\Delta,\mathrm{Aff}_{\Z,\Delta})$.

\begin{defin}
A \emph{piecewise \Zaffine structure} on a polyhedral complex $\Sigma$ consists of an \Zaffine structure for each polyhedron in $\Sigma$, such that if $\sigma$ is a face of $\tau$, then the restriction of the \Zaffine structure on $\tau$ to $\sigma$ equals the \Zaffine structure on $\sigma$, in the sense that $\mathrm{Aff}_{\Z,\tau}|_\sigma\simeq\mathrm{Aff}_{\Z,\sigma}$.
\end{defin}

\begin{eg}
The embeddings of the Clemens cone $\CsXH$ and the Clemens polytope $\SsXH$ into $\R^{\Iv\cup\Ih}$ induce naturally piecewise \Zaffine structures on $\CsXH$ and $\SsXH$.
\end{eg}

\begin{defin}\label{def:polyhedral_Zaffine_manifold_with_sigularities}
A \emph{polyhedral \Zaffine manifold with singularities} consists of a polyhedral complex $\Sigma$ equipped with a piecewise \Zaffine structure, an open subset $\Sigma_0\subset\Sigma$ which is a manifold without boundary, and a \Zaffine structure $\mathrm{Aff}_{\Z,\Sigma_0}$ on $\Sigma_0$ compatible with the piecewise \Zaffine structure on $\Sigma$, in the sense that the restriction of $\mathrm{Aff}_{\Z,\Sigma_0}$ to the intersection between $\Sigma_0$ and each polyhedron in $\Sigma$ is isomorphic to the one given by the piecewise \Zaffine structure on $\Sigma$.
The points in the set $\Sigma_0$ are called \emph{smooth points} with respect to the \Zaffine structure,
while points in the set $\Sigma\setminus\Sigma_0$ are called \emph{singular points}.
\end{defin}

\section{Log Calabi-Yau surfaces} \label{sec:log-CY_surfaces}

\begin{defin}[cf.\ \cite{Gross_Mirror_Log_2011}]\label{def:Looijenga pair}
A \emph{Looijenga pair} $(Y,D)$ consists of a connected smooth complex projective surface $Y$ together with a singular nodal curve $D$ representing the anti-canonical class $-K_Y$.
\end{defin}

We note that the definition of Looijenga pair implies that $Y$ is a rational surface, and that $D$ is either an irreducible arithmetic genus one curve with a single node, or a cycle of smooth rational curves.
For simplicity, we will assume that $D$ is a cycle of at least three smooth rational curves.
This can always be achieved by blowing up the nodes, which does not change the geometry we study.
We order the irreducible components of $D$ cyclically, and write $D=D_1+\dots+D_l$.
We take indices modulo $l$.

Let $(Y,D)$ be a Looijenga pair.
Let $X=Y\setminus D$.
Let $Y_{\kc}=Y\times_{\Spec\C}\Spec\kc$, $D_{\kc}=D\times_{\Spec\C}\Spec\kc$, $Y_k=Y\times_{\Spec\C}\Spec k$, $D_k=D\times_{\Spec\C}\Spec k$, and $X_k=X\times_{\Spec\C}\Spec k$.
We have $(Y_{\kc}\setminus D_{\kc})_\eta \simeq X_k$.
So $(Y_{\kc}, D_{\kc})$ is an snc log-model of $X_k$ in the sense of \cref{def:snc_log_model}.

Let $(\be',\be_1,\be_2,\dots,\be_l)$ be the standard basis of $\R^{l+1}$
and let $(a',a_1,a_2,\dots,a_l)$ be the coordinates.

By definition, we have
\[C_{(Y_{\kc},D_{\kc})} = \bigcup_{i=1}^l\Set{a'\be'+a_i\be_i+a_{i+1}\be_{i+1} | a', a_i, a_{i+1}\ge 0} \subset \R^{l+1}.\]

We have an isomorphism $\R^{l+1}\cap\{a'=1\}\simeq\R^l$.
By abuse of notation, let $(\be_1,\dots,\be_l)$ and $(a_1,\dots,a_l)$ denote respectively the induced basis and the induced coordinates on $\R^l$.
We have
\begin{align*}
S_{(Y_{\kc},D_{\kc})} &= C_{(Y_{\kc},D_{\kc})} \cap \{a'=1\}\\
&\simeq \bigcup_{i=1}^l \Set{a_i\be_i+a_{i+1}\be_{i+1} | a_i,a_{i+1}\ge 0}\subset\R^l.
\end{align*}
We denote $B=S_{(Y_{\kc},D_{\kc})}$ for simplicity.
We call $B$ the \emph{tropical base} associated to the Looijenga pair $(Y,D)$.
We have the retraction map $\tau\colon X_k\an\to B$ constructed in \cref{sec:review_tropicalizations}.

Set
\begin{align*}
\Delta_i&\coloneqq\Set{a_i\be_i | a_i\ge 0}\subset B,\\
\Delta_{i,i+1}&\coloneqq\Set{a_i\be_i+a_{i+1}\be_{i+1}|a_i,a_{i+1}\ge 0}\subset B.
\end{align*}
Their relative interiors are denoted by $\Delta^\circ_i$ and $\Delta^\circ_{i,i+1}$.
We denote by $O\in B$ the point corresponding to the origin in $\R^l$.

Let $\val^n$ denote the map
\[\val^n\colon \Gmknan\to\R^n\qquad(x_1,\dots,x_n)\mapsto(\val x_1,\dots,\val x_n).\]

\begin{defin}
A continuous map from a \kanal space to a topological space is called an \emph{affinoid torus fibration} if locally on the target, it is of the form $(\val^n)\inv(U)\to U$ for some open subset $U\subset\R^n$.
\end{defin}

\begin{rem}
Affinoid torus fibrations are analogous to Lagrangian torus fibrations in symplectic geometry.
\end{rem}

By construction, the retraction map $\tau\colon X_k\an\to B$ is an affinoid torus fibration over $\bigcup_{i=1}^l\Delta_{i,i+1}^\circ$ (cf.\ \cite[\S 4.2]{Gubler_Skeletons_2014}).
The aim of the rest of this section is to extend this affinoid torus fibration over codimension one open strata $\Delta_i^\circ$.

Let $C$ denote the complex projective line $\bbP^1_\C$ with a chosen coordinate.
Let $T_n$ denote the total space of the line bundle $\cO(-n)$ on $C$ for any integer $n$.
Let $T_{n,0}$ and $T_{n,\infty}\subset T_n$ denote respectively the fibers at $0$ and $\infty$.
Let $(C_n,E_n,F_n)$ denote the formal completion of $(T_n, T_{n,0}, T_{n,\infty})$ along the zero section.
Let $p\in C\setminus\{0,\infty\}$.

Let $\widetilde C_n$ denote the blowup of $C_n$ at the point $p$.
Let $\hC_n$ denote the formal completion of $\widetilde C_n$ along the strict transform of $C$.
We have a natural morphism $\hC_n\to C_n$.
Let $\hE_n$ and $\hF_n$ denote respectively the pullback of the divisors $E_n$ and $F_n$.

\begin{lem}\label{lem:blowup}
The triple $(\hC_n, \hE_n, \hF_n)$ is isomorphic to the triple $(C_{n+1},E_{n+1},F_{n+1})$.
\end{lem}
\begin{proof}
Let $U_{n,0}$ be the affine formal scheme $C_n\setminus\{\infty\}$ and let $U_{n,\infty}$ be $C_n\setminus\{0\}$.
Assume
\[ U_{n,0}\simeq\Spf\C[x_0]\llb y_0\rrb,\qquad U_{n,\infty}\simeq\Spf\C[x_\infty]\llb y_\infty\rrb.\]
The gluing of $U_{n,0}$ and $U_{n,\infty}$ is given by
\begin{align*}
\C[x_0,x_0\inv]\llb y_0\rrb &\longrightarrow\C[x_\infty,x\inv_\infty]\llb y_\infty\rrb\\
x_0 &\longmapsto x\inv_\infty\\
x\inv_0 &\longmapsto x_\infty\\
y_0 &\longmapsto x^n_\infty\cdot y_\infty.
\end{align*}

Similarly, let $\hU_{n,0}$ be the affine formal scheme $\hC_n\setminus\{\infty\}$ and let $\hU_{n,\infty}$ be $\hC_n\setminus\{0\}$.
We choose coordinates \[u_0=\frac{y_0}{x_0-p},\qquad u_\infty=\frac{y_\infty}{1-p x_\infty}.\]
We obtain
\[ \hU_{n,0}\simeq\Spf\C[x_0]\llb u_0\rrb,\qquad \hU_{n,\infty}\simeq\Spf\C[x_\infty]\llb u_\infty\rrb.\]
The gluing of $\hU_{n,0}$ and $\hU_{n,\infty}$ is given by
\begin{align*}
\C[x_0,x_0\inv]\llb u_0\rrb &\longrightarrow\C[x_\infty,x\inv_\infty]\llb u_\infty\rrb\\
x_0 &\longmapsto x\inv_\infty\\
x\inv_0 &\longmapsto x_\infty\\
u_0 &\longmapsto \frac{y_0}{x_0-p} = \frac{x^n_\infty\cdot y_\infty}{\frac{1}{x_\infty}-p}=\frac{x^{n+1}_\infty\cdot y_\infty}{1- p x_\infty}=x^{n+1}_\infty\cdot u_\infty.
\end{align*}

Therefore, we obtain an isomorphism $\hC_n\simeq C_{n+1}$.
Since both divisors $\hE_n$ and $E_{n+1}$ are defined by the equation $x_0=0$,
and both divisors $\hF_n$ and $F_{n+1}$ are defined by the equation $x_\infty=0$,
the isomorphism $\hC_n\simeq C_{n+1}$ induces an isomorphism $(\hC_n, \hE_n, \hF_n)\simeq(C_{n+1},E_{n+1},F_{n+1})$.
\end{proof}

\begin{rem}\label{rem:toric_description}
We observe that the triple $(C_n,E_n,F_n)$ has a toric description.
Let $u=(1,0), v=(0,1), w=(-1,n)$ be three vectors in $\Z^2$.
Consider the fan in $\R^2$ consisting of a cone generated by the vectors $u,v$ and another cone generated by the vectors $v,w$.
Let $Z$ be the corresponding smooth quasi-projective toric surface.
Let $D_u, D_v, D_w$ be the toric divisors corresponding to the rays in the directions $u,v,w$.
The divisor $D_v$ is isomorphic to a projective line.
Its normal bundle is isomorphic to the line bundle $\cO(-n)$.
Let $(C',E',F')$ denote the formal completion of $(Z,D_u,D_w)$ along $D_v$.
Then the triple $(C',E',F')$ is isomorphic to the triple $(C_n,E_n,F_n)$.
\end{rem}

\begin{prop}\label{prop:extension_of_torus_fibration}
The retraction map $\tau\colon X_k\an\to B$ is an affinoid torus fibration over $B\setminus O$, where $O$ denotes the origin of $B$.
\end{prop}
\begin{proof}
Since the retraction map $\tau\colon X_k\an\to B$ is an affinoid torus fibration over $\bigcup_{i=1}^l\Delta_{i,i+1}^\circ$,
it suffices to show that for any $i\in\{1,\dots,l\}$, any point $b$ in the open stratum $\Delta^\circ_i$, the retraction map $\tau\colon X_k\an\to B$ is an affinoid torus fibration over a neighborhood of $b$.

The toroidal construction of \cite{Kempf_Toroidal_1973} gives a one-to-one correspondence between finite rational conical subdivisions of the Clemens polytope $B$ and toric blowups of the Looijenga pair $(Y,D)$ in the sense of \cite[\S 1.3]{Gross_Mirror_Log_2011}.
In particular, toric blowups of the pair $(Y,D)$ do not alter the retraction map $\tau\colon X_k\an\to B$ near the point $b\in B$.
So we can assume that $(Y,D)$ admits a toric model $(\oY,\oD)$ by \cite[Proposition 1.19]{Gross_Mirror_Log_2011}, i.e.\ the pair $(Y,D)$ is obtained from a smooth projective toric surface $(\oY,\oD)$ by blowing up finitely many points on the smooth locus of $\oD$.

Let $(C',E',F')$ denote the formal completion of the triple $(Y,D_{i-1},D_{i+1})$ along the divisor $D_i$.
By \cref{lem:blowup} and \cref{rem:toric_description}, it is thus isomorphic to the triple $(C_n,E_n,F_n)$ for some $n$.
By construction, the retraction map $\tau\colon X_k\an\to B$ over the open subset 
\begin{equation}\label{eq:cone}
\Delta_{i-1,i}^\circ\cup\Delta_i^\circ\cup\Delta_{i,i+1}^\circ\subset B
\end{equation}
is completely determined by the triple $(C',E',F')$.
Moreover, since we have affinoid torus fibration everywhere on the base in the toric case, by \cref{rem:toric_description} again, the retraction map $\tau\colon X_k\an\to B$ is an affinoid torus fibration over the open subset in \cref{eq:cone}.
\end{proof}

Let $\psi_i\colon\Delta_{i-1,i}^\circ\cup\Delta_i^\circ\cup\Delta_{i,i+1}^\circ\to\R^2$ be the continuous map which is linear on $\Delta_{i-1,i}^\circ$, $\Delta_i^\circ$ and $\Delta_{i,i+1}^\circ$, and which sends the vectors $\be_{i-1}, \be_i, \be_{i+1}$ respectively to the vectors $(1,0), (0,1), (-1,-D^2_i)\in\R^2$, where $D_i^2$ denote the self-intersection number of the curve $D_i$ inside the surface $Y$.

\begin{rem}\label{rem:Zaffine_structure_on_B}
By \cref{prop:extension_of_torus_fibration} and \cite[Theorem 1]{Kontsevich_Affine_2006}, we obtain a \Zaffine structure on $B\setminus O$ so that the tropical base $B$ is a polyhedral \Zaffine manifold with a singularity at $O$ in the sense of \cref{def:polyhedral_Zaffine_manifold_with_sigularities}.
By the proof of \cref{prop:extension_of_torus_fibration}, the \Zaffine structure restricted to $\Delta_{i-1,i}^\circ\cup\Delta_i^\circ\cup\Delta_{i,i+1}^\circ$ is isomorphic to the pullback of the standard \Zaffine structure on $\R^2$ by the map $\psi_i$ defined above.
So we obtain the same \Zaffine structure as in \cite[\S 1.1]{Gross_Mirror_Log_2011}.
By \cite[Lemma 1.3]{Gross_Mirror_Log_2011}), the \Zaffine structure on $B\setminus O$ extends across the origin if and only if $Y$ is toric and $D$ is the toric boundary.
\end{rem}

\section{Tropical cylinders} \label{sec:tropical_cylinders}

Tropical geometry is a fundamental tool in our enumeration of holomorphic cylinders.
In this section, we introduce the notion of spines, tropical cylinders, and extensions of them.
We prove a rigidity property of tropical cylinders, which will be an important ingredient in the proof of \cref{thm:properness_of_moduli}.

\begin{defin}
An \emph{unbounded \Zaffine graph} $(\Gamma, V_\infty(\Gamma))$ consists of the following data:
\begin{enumerate}[(i)]
\item A finite graph $\Gamma$ without loops.
We denote by $V(\Gamma)$ the set of vertices of $\Gamma$ and by $E(\Gamma)$ the set of edges of $\Gamma$.
\item A subset of 1-valent vertices $V_\infty(\Gamma)\subset V(\Gamma)$ called \emph{unbounded vertices}.
We call the rest of the vertices \emph{bounded vertices}.
\item For every edge $e$ with two bounded endpoints, a \Zaffine structure on $e$ which is isomorphic to the closed interval $[0,\alpha]\subset\R$ for a positive real number $\alpha$.
\item For every edge $e$ with one unbounded vertex $v_\infty$, a \Zaffine structure on $e\setminus\{v_\infty\}$ which is isomorphic to the interval $[0,+\infty)\subset\R$.
\item For every edge $e$ with two unbounded vertices $v_\infty$ and  $v'_\infty$, a \Zaffine structure on $e\setminus\{v_\infty, v'_\infty\}$ which is isomorphic to the standard \Zaffine structure on $\R$.
\end{enumerate}
We will simply say a \emph{\Zaffine graph} where unbounded vertices are not present.
\end{defin}

\begin{defin}
An (unbounded) \Zaffine tree is an (unbounded) \Zaffine graph whose underlying graph is a tree.
\end{defin}

\begin{defin} \label{def:Zaffine_map}
	Let $(\Gamma,V_\infty(\Gamma))$ be an unbounded \Zaffine graph and let $\Sigma$ be a polyhedral complex with a piecewise \Zaffine structure.
	A \emph{\Zaffine map} $h\colon\Gamma\setminus V_\infty(\Gamma) \to \Sigma$ is a continuous proper map such that every edge maps to a polyhedron in $\Sigma$ and the maps are compatible with the \Zaffine structures.
A \emph{\Zaffine immersion} is a \Zaffine map that does not contract any edge to a point.
\end{defin}

Let $\Sigma$ be a polyhedral \Zaffine manifold with singularities and let $\Sigma_0\subset\Sigma$ be the set of smooth points in the sense of \cref{def:polyhedral_Zaffine_manifold_with_sigularities}.
Let $h$ be a \Zaffine map from an unbounded \Zaffine graph $(\Gamma, V_\infty(\Gamma))$ to $\Sigma$.
For a bounded vertex $v\in V(\Gamma)$ which maps to $\Sigma_0$ and an edge $e$ connected to $v$,
we denote by $w^0_v(e)$ the unitary integral tangent vector at $v$ pointing to the direction of $e$, and by $w_v(e)$ the image of $w^0_v(e)$ in $T_{h(v)}\Sigma(\Z)$, where $T_{h(v)}\Sigma(\Z)$ denotes the integral lattice in the tangent space of the \Zaffine manifold $\Sigma_0$ at the point $h(v)$.

\begin{defin}
The \Zaffine map $h$ above is said to be \emph{balanced} at a vertex $v$ which maps to $\Sigma_0$ if we have $\sum_{e\ni v} w_v(e)=0\in T_{h(v)}\Sigma(\Z)$.
\end{defin}

From now on we restrict to the particular polyhedral \Zaffine manifold $B$ with a singularity at $O$ in \cref{rem:Zaffine_structure_on_B}.

\begin{defin} \label{def:spine}
A \emph{spine} in the tropical base $B$ consists of a connected \Zaffine tree $\Gamma$, two 1-valent vertices $(v_1,v_2)$ of $\Gamma$, and a \Zaffine immersion $h\colon\Gamma\to B$ satisfying the following conditions:
\begin{enumerate}[(i)]
\item The image of $h$ does not contain the origin $O\in B$.
\item The vertices $v_1$ and $v_2$ are the only 1-valent vertices of $\Gamma$.
\item For every vertex $v$, every edge $e$ connected to $v$, neither of the vectors $\pm w_v(e)$ points towards the origin $O\in B$. 
\item \label{item:spine:balancing} For every 2-valent vertex $v$, the vector $-\sum_{e\ni v} w_v(e)$ is either zero or points towards to origin $O\in B$.
\end{enumerate}
\end{defin}

\begin{defin} \label{def:tropical_cylinder}
A \emph{tropical cylinder} in the tropical base $B$ consists of a connected \Zaffine tree $\Gamma$, a pair of two 1-valent vertices $(v_1, v_2)$ of $\Gamma$, and a \Zaffine immersion $h\colon\Gamma\to B$ satisfying the following conditions:
\begin{enumerate}[(i)]
\item The inverse image of the origin $O\in B$ consists of all the 1-valent vertices of $\Gamma$ except $v_1$ and $v_2$.
\item For every vertex $v$ whose valency is greater than 1, the \Zaffine map $h$ is balanced at $v$.
\end{enumerate}
\end{defin}

\begin{defin}
An \emph{extended spine} in the tropical base $B$ consists of a connected unbounded \Zaffine tree $\Gamma$ with a pair of unbouded vertices $(v_1,v_2)$  and a \Zaffine immersion $h\colon\Gamma\setminus\{v_1,v_2\}\to B$ such that Conditions (i)-(\ref{item:spine:balancing}) of \cref{def:spine} hold.
\end{defin}

\begin{defin} \label{def:extended_tropical_cylinder_in_B}
An \emph{extended tropical cylinder} in the tropical base $B$ consists of a connected unbounded \Zaffine tree $\Gamma$ with a pair of unbounded vertices $(v_1,v_2)$ and a \Zaffine immersion $h\colon\Gamma\setminus\{v_1,v_2\}\to B$ such that Conditions (i) and (ii) of \cref{def:tropical_cylinder} hold.
\end{defin}

We endow $\R$ with the standard \Zaffine structure $\mathrm{Aff}_{\Z,\R}$.
We obtain a \Zaffine structure on the product $\R\times(B\setminus O)$, so $\R\times B$ is a polyhedral \Zaffine manifold with singularities.
Set $\tB\coloneqq\R\times B$.
We denote by $\pi_1\colon\tB\to \R$ and $\pi_2\colon\tB\to B$ the two projections.

\begin{rem}\label{rem:advantage_of_tilde}
The factor $\R$ in $\tB$ will have two advantages in our approach:
\begin{enumerate}[(i)]
\item It will allow us to treat non-injective spines in a simple way by considering their graphs.
\item It will allow us to eliminate certain multiplicities. Consequently, we can circumvent the sophisticated machinery of relative Gromov-Witten invariants (cf.\ \cite{Li_Symplectic_surgery_2001,Ionel_Relative_Gromov-Witten_2003,Gathmann_Absolute_2002,Li_Stable_morphisms_2001,Li_A_degeneration_formula_2002,Gross_Logarithmic_2013}).
\end{enumerate}
\end{rem}

\begin{defin} \label{def:extended_tropical_cylinder_in_tB}
An \emph{extended tropical cylinder} in $\tB$ consists of a connected unbounded \Zaffine tree $\Gamma$ with a pair of unbounded vertices $(v_1,v_2)$ and a \Zaffine immersion $h\colon\Gamma\setminus\{v_1,v_2\}\to \tB$ such that
\begin{enumerate}[(i)]
\item the composition $\pi_1\circ h$ is a \Zaffine map balanced at every vertex of $\Gamma$, and
\item the composition $\pi_2\circ h$ is an extended tropical cylinder in $B$ as in \cref{def:extended_tropical_cylinder_in_B}.
\end{enumerate}
\end{defin}

Let $L=(\Gamma,(v_1,v_2),h\colon \Gamma\to B)$ be a spine in $B$.
We define in the following an extension $L'=(\Gamma',\{v'_1,v_2\},h')$ of the spine $L$ and a curve class $\beta'\in\NE(Y)$ associated to the extension.
We set initially $\Gamma'=\Gamma$ and $h'=h$.
Let $e_1$ denote the edge connected to the vertex $v_1$.
Let $r_1$ denote the ray starting at $h(v_1)$ with direction opposite to $h(e_1)$.
We distinguish two cases according to whether the ray $r_1$ hits, or not, a codimension one face $\Delta_i$ of $B$.
\begin{enumerate}[(i)]
\item Assume that the ray $r_1$ does not hit any codimension one face of $B$ except possibly the point $h'(v_1)$.
Then we add an unbounded vertex $v'_1$ to $\Gamma'$ and an edge $e'_1$ connecting $v'_1$ with $v_1$.
We endow $e'_1\setminus\{v'_1\}$ with the \Zaffine structure isomorphic to the interval $[0,+\infty)\subset\R$.
Let $h'|_{e'_1}$ be the \Zaffine map sending $e'_1\setminus\{v'_1\}$ to the ray $r_1$ so that $h'$ becomes balanced at the vertex $v_1$.
In this case, we set $\beta'=0$.
\item Otherwise, we add a bounded vertex $v'_1$ to $\Gamma'$ and an edge $e'_1$ connecting $v'_1$ with $v_1$.
Let $h'(v'_1)$ be the first point other than $h'(v_1)$ where the ray $r_1$ hits some codimension one face $\Delta_i\subset B$.
Let $\widebar e'_1$ denote the segment connecting $h'(v_1)$ and $h'(v'_1)$.
Assume that the pullback of the piecewise \Zaffine structure on $B$ to $\widebar e'_1$ is isomorphic to the standard \Zaffine structure on a closed interval $[0,\alpha]$ for a positive real number $\alpha$.
Let $m$ be the ratio $w_{v_1}(e_1)/w^0_{v_1}(e_1)$.
We endow the edge $e'_1$ with the \Zaffine structure isomorphic to the interval $[0,\alpha/m]\subset\R$.
Let $h'|_{e'_1}$ be the \Zaffine map sending the edge $e'_1$ to the segment $\widebar e'_1$ so that $h'$ becomes balanced at the vertex $v_1$.
Let $\be_i$ be the primitive integral vector in the direction of $\Delta_i$.
Let $\mu$ be the lattice length of the wedge product $w_{v'_1}(e'_1)\wedge\be_i$.
We set $\beta'=\mu[D_i]\in\NE(Y)$.

\end{enumerate}
We call the triple $(\Gamma',\{v'_1,v_2\},h')$ the \emph{extension} of the spine $L$ at the vertex $v_1$,
and we call $\beta'$ the curve class associated to this extension.
Similarly, we can extend the spine $L$ at the vertex $v_2$.
We do iterated extensions at both sides until both sides end with unbounded vertices.
If this happens after finitely many steps, we call the spine $L$ \emph{extendable}, we call the final product the \emph{extended spine associated to the spine} $L$, and we define the curve class associated to this extension to be the sum of all the curve classes associated to the intermediate extensions.

\begin{rem}
The Looijenga pair is called \emph{positive} if the intersection matrix $(D_i\cdot D_j)$ is not negative semi-definite.
This holds if and only if $Y\setminus D$ is the minimal resolution of an affine surface (cf.\ \cite[Lemma 5.9]{Gross_Mirror_Log_2011}).
In this case, every spine in the tropical base $B$ is extendable (cf.\ \cite[Corollary 1.6]{Gross_Mirror_Log_2011}).
\end{rem}

Let $\hL=(\Gamma,(v_1,v_2),h)$ be an extended spine in $B$.
Set $\Gamma'=\Gamma$ and $h'=h$.
For every bounded vertex $v$ of $\Gamma$ such that $\sigma_v\coloneqq\sum_{e\ni v} w_v(e)$ is non-zero, we add a vertex $v'$ to $\Gamma'$ and an edge $e_v$ connecting $v$ and $v'$.
Set $h'(v')=O$.
Let $\sigma_v^0$ denote the primitive vector in $T_{h'(v)} B(\Z)$ in the direction of $\sigma_v$.
Let $m_v$ be the ratio $\sigma_v/\sigma_v^0$.
Let $\widebar e_v$ denote the segment connecting $h'(v)$ and $h'(v')$.
Assume that the pullback of the piecewise \Zaffine structure on $B$ to $\widebar e_v$ is isomorphic to the standard \Zaffine structure on a closed interval $[0,\alpha_v]$ for a positive real number $\alpha_v$.
We endow the edge $e_v$ with the \Zaffine structure isomorphic to the interval $[0,\alpha_v/m_v]\subset\R$.
Let $h'|_{e_v}$ be the \Zaffine map sending the edge $e_v$ to the segment $\widebar e_v$ so that $h'$ becomes balanced at the vertex $v$.
The resulting triple $(\Gamma', (v_1,v_2), h')$ is an extended tropical cylinder in $B$, which we call the \emph{extended tropical cylinder in $B$ associated to the extended spine $\hL$}.

Since $\hL=(\Gamma,(v_1,v_2),h)$ is an extended spine in $B$, we see that $\Gamma\setminus\{v_1,v_2\}$ is homeomorphic to the real line $\R$.
Let $\mathrm{Aff}_{\Z,\Gamma}$ be the \Zaffine structure on $\Gamma\setminus\{v_1,v_2\}$ whose restriction to every edge of $\Gamma$ coincides with the given \Zaffine structure on the edge.
We obtain an isomorphism of \Zaffine manifolds $h_0\colon(\Gamma\setminus\{v_1,v_2\},\mathrm{Aff}_{\Z,\Gamma})\to(\R,\mathrm{Aff}_{\Z,\R})$.
Let $\pi_\Gamma\colon\Gamma'\to\Gamma$ be the map contracting all newly added edges.
Set
\[h''=(h_0\circ\pi_\Gamma\colon\Gamma'\setminus\{v_1,v_2\}\to\R,\  h'\colon\Gamma'\setminus\{v_1,v_2\}\to B).\]
Then the triple $(\Gamma',(v_1,v_2),h'')$ is an extended tropical cylinder in $\tB$, which we call the \emph{extended tropical cylinder in $\tB$ associated to the extended spine $\hL$}.

\subsection*{Rigidity of the tropical cylinder}

Using the embedding $B\subset\R^l$, for any point $p\in B\setminus O$, any vector $w\in T_p B(\Z)$, we denote by $(w^1,\dots,w^l)$ the coordinates of $w$ with respect to the basis $(\be_1,\dots,\be_l)$ of $\Z^l\subset\R^l$.
We define $|w|=\sqrt{(w^1)^2+\dots+(w^l)^2}$, called the \emph{norm} of the vector $w$.

Let $Z=(\Gamma,(v_1,v_2),h)$ be an extended tropical cylinder in $\tB$.
Let $e_1$ and $e_2$ denote the edges connected to the vertices $v_1$ and $v_2$ respectively.
Let $\widebar v_1$ be the bounded endpoint of $e_1$ and let $\widebar v_2$ be the bounded endpoint of $e_2$.
We denote $A(Z)\coloneqq\max\{|w_{\widebar v_1}(e_1)|,|w_{\widebar v_2}(e_2)|\}$.

\begin{prop}\label{prop:rigidity_of_tropical_cylinder}
Let $Z=(\Gamma,(v_1,v_2),h)$ be an extended tropical cylinder in $\tB$.
Let $e_1$ and $e_2$ denote the edges connected to the vertices $v_1$ and $v_2$ respectively.
Let $p$ be a point on the ray $h(e_1\setminus\{v_1\})$.
Let $A_0$ be a real number such that $A_0>A(Z)$.
Then there exists an open neighborhood $U$ of the image $h(\Gamma)$ in $\tB$ satisfying the following conditions:
\begin{enumerate}[(i)]
\item Let $Z'=(\Gamma',\{v_1',v_2'\},h')$ be another extended tropical cylinder in $\tB$ such that $A(Z')<A_0$ and that its image $h'(\Gamma')$ lies in $U$ and contains $p$, then the image $h'(\Gamma')$ coincides necessarily with the image $h(\Gamma)$.
\item \label{item:rigidity:polyhedral} For $i=1,2$, there exists a point $p_i$ on the ray $h(e_i\setminus\{v_i\})$ and a neighborhood $U_i$ of $p_i$ in $\tB$, such that any translation of $U_i$ along the ray $h(e_i\setminus\{v_i\})$ is contained in $U$.
\end{enumerate}
\end{prop}
\begin{proof}
Let $v_1,v_{11},v_{12},\dots,v_{1m},v_2$ be the chain of vertices on the path connecting the vertices $v_1$ and $v_2$ in $\Gamma$.
Let $e_{1j}$ denote the edge connecting the vertices $v_{1j}$ and $v_{1,(j+1)}$.
For any extended tropical cylinder $Z'=(\Gamma',\{v_1',v_2'\},h')$ in $\tB$, we denote by $e'_1$ and $e'_2$ the edges connected to the vertices $v'_1$ and $v'_2$ respectively.
We introduce the similar notations $v'_1,v'_{11},v'_{12},\dots,v'_{1m'},v'_2$ and $e'_{11}, e'_{12}, \dots$ for $Z'$.

We start with any neighborhood $U$ of the image $h(\Gamma)$ in $\tB$ satisfying Condition (ii).
By shrinking $U$ near the rays $h(e_1\setminus\{v_1\})$ and $h(e_2\setminus\{v_2\})$, we can assume that for any extended tropical cylinder $Z'=(\Gamma',\{v_1',v_2'\},h')$ in $\tB$ such that $h'(\Gamma')\subset U$,
the rays $h(e'_1\setminus\{v'_1\})$ and $h(e'_2\setminus\{v'_2\})$ are necessarily parallel to the rays $h(e_1\setminus\{v_1\})$ and $h(e_2\setminus\{v_2\})$, up to a switch of the labeling of $v'_1$ and $v'_2$.

Assume moreover that $A(Z')<A_0$.
By further shrinking $U$, we can ensure that the ray $O\pi_2(h'(v'_{11}))$ coincides with the ray $O\pi_2(h(v_{11}))$.
By further shrinking $U$, we can ensure that the segment $\pi_2(h'(e'_{11}))$ is parallel to the segment $\pi_2(h(e_{11}))$.
Since the norm $| w_{v'_{11}}(e'_{1}) |$ is bounded by $A_0$, the balancing condition at the vertex $v'_{11}$ implies that the norm $| w_{v'_{11}}(e'_{11}) |$ is bounded by a number depending only on $A_0$ and $Z$.
Therefore, by further shrinking $U$, we can ensure that the ray $O\pi_2(h'(v'_{12}))$ coincides with the ray $O\pi_2(h(v_{12}))$ and that the segment $\pi_2(h'(e'_{12}))$ is parallel to the segment $\pi_2(h(e_{12}))$.
Iterating the process, we can ensure that $m=m'$, that the ray $O\pi_2(h'(v'_{1j}))$ coincides with the ray $O\pi_2(h(v_{1j}))$ for $j=1,\dots,m$, and that the segment $\pi_2(h'(e'_{1j}))$ is parallel to the segment $\pi_2(h(e_{1j}))$ for $j=1,\dots,m-1$.

Assume moreover that the image $h'(\Gamma')$ contains the point $p$.
Then the image $\pi_2(h'(\Gamma'))$ contains the point $\pi_2(p)$.
By further shrinking $U$ if necessary, we can ensure that the point $\pi_2(h'(v'_{1j}))$ coincides with the point $\pi_2(h(v_{1j}))$ for $j=1,\dots,m$, and that the segment $\pi_2(h'(e'_{1j}))$ coincides with the segment $\pi_2(h(e_{1j}))$ for $j=1,\dots,m-1$.
By further shrinking $U$ if necessary, the balancing condition at the vertices $v'_{1j}$ ensures that the point $h'(v'_{1j})$ coincides with the point $h(v_{1j})$ for $j=1,\dots,m$, and that the segment $h'(e'_{1j})$ coincides with the segment $h(e_{1j})$ for $j=1,\dots,m-1$.
Finally, the conditions in \cref{def:extended_tropical_cylinder_in_tB} ensures that the image $h'(\Gamma')$ coincides with the image $h(\Gamma)$.
We remark that at every stage of shrinking, we can always ensure that Condition (\ref{item:rigidity:polyhedral}) in the proposition holds.
\end{proof}

Let $C_{\tB}$ denote $\R_{\ge 0}\times \tB$.
We embed $\tB$ into $C_{\tB}$ by the map $x\mapsto (1,x)$.
Let $C'_{\tB}$ be a finite subdivision of $C_{\tB}$ into rational polyhedral cones, so that it induces a polyhedral subdivision $\tB'$ of $\tB$.
For any one-dimensional cone $r$ in $C'_{\tB}$, we denote by $\Star(r)$ the union of the relative interior of the polyhedral cones in $C'_{\tB}$ containing $r$, and by $\oStar(r)$ the union of the polyhedral cones in $C'_{\tB}$ containing $r$.

For a set $R$ of one-dimensional cones in $C'_{\tB}$, we denote
\[U(R)\coloneqq\bigg(\bigcup_{r\in R} \Star(r)\bigg)\cap \tB',\qquad \oU(R)\coloneqq\bigg(\bigcup_{r\in R} \oStar(r)\bigg)\cap \tB'.\]

\begin{rem}\label{rem:rigidity_polyhedral}
Condition (\ref{item:rigidity:polyhedral}) of \cref{prop:rigidity_of_tropical_cylinder} implies that there exists a regular simplicial subdivision $C'_{\tB}$ of $C_{\tB}$ and a set $R$ of one-dimensional cones in $C'_{\tB}$ intersecting $\tB$ such that $h(\Gamma)\subset U(R)\subset \oU(R)\subset U$.
\end{rem}

\section{Holomorphic cylinders} \label{sec:hol_cylinders}

Given an extendable spine $L=(\Gamma_0,(v_{10},v_{20}),h_0)$ in the tropical base $B$ and an element $\beta$ in the cone of curves $\NE(Y)$,
the goal of this section is to define a rational number $N(L,\beta)$, which we call \emph{the number of holomorphic cylinders associated to $L$ with class $\beta$}.

Let $L=(\Gamma_0,(v_{10},v_{20}),h_0)$ be an extendable spine in the tropical base $B$.
Let $\hL=(\hGamma_0,(v_1,v_2),\hh_0)$ be the extended spine in $B$ associated to $L$.
Let $\beta'$ be the curve class associated to the extension and let $\hbeta\coloneqq\beta+\beta'$.
Let $Z=(\Gamma,(v_1,v_2),h\colon\Gamma\setminus\{v_1,v_2\}\to\tB)$ be the extended tropical cylinder in $\tB$ associated to the extended spine $\hL$.
Let $e_1,e_2$ denote the edges connected to the vertices $v_1,v_2$ respectively.

Choose any real number $A_0>A(Z)$.
Let $U$ be the open neighborhood constructed in \cref{prop:rigidity_of_tropical_cylinder}.
Let $C'_{\tB}$ and $R$ be as in \cref{rem:rigidity_polyhedral}.
We can assume that there are two rays $\{r_1,r_2\}\subset R$ pointing in the directions of $h(e_1\setminus\{v_1\})$ and $h(e_2\setminus\{v_2\})$ respectively.

Let $\tX=\bbG_{\mathrm m/\C}\times(Y\setminus D)$ and $\tX_k=\tX\times_{\Spec\C}\Spec k$.
Then
\[\Big((\bbP_\C^1\times Y)\times_{\Spec\C}\Spec\kc,\ \big((\{0,\infty\}\times Y) \cup (\bbP_\C^1\times D)\big)\times_{\Spec\C}\Spec\kc\Big)\]
is an snc log-model of $\tX_k$.
The factor $\bbG_{\mathrm m/\C}$ in $\tX$ corresponds to the factor $\R$ in $\tB$ (see \cref{rem:advantage_of_tilde}).
By toroidal modification, the regular simplicial subdivision $C'_{\tB}$ induces an snc log-model $(\tsX,\tH)$ of $\tX_k$.
Let $\tY$ denote the $k$-variety $\big(\tsX\big)_\eta$.
Let $\tD_1$, $\tD_2$ denote the divisors in $\tY$ corresponding to the rays $r_1, r_2$ respectively.

Let $\tB'_f$ denote the union of bounded polyhedral faces of $\tB'$.
It is by construction the Clemens polytope associated to the snc model $\tsX$ of $\tY$.
We have a retraction map $\ttau\colon\tY\an\to\tB'_f$.
Let $U_f=U\cap\tB'_f$, $U(R)_f=U(R)\cap\tB'_f$, $\oU(R)_f=\oU(R)\cap\tB'_f$.

\begin{defin}\label{def:class_beta}
An element $\gamma\in\NE(\tY)$ is said to \emph{represent} the class $\hbeta\in\NE(Y)$ if the following conditions are satisfied:
\begin{enumerate}[(i)]
\item The element $\gamma$ admits a representative which is a closed curve $C$ contained in $\tX_k\cup\tD_1\cup\tD_2$ intersecting $\tD_1\cup\tD_2$ transversely.
\item \label{item:class_beta:intersection} We require that the intersection numbers $\gamma\cdot \tD_1=1$, $\gamma\cdot \tD_2=1$.
\item \label{item:class_beta:degree_one} Under the projection $\NE(\tY)\to\NE(\bbP^1_k)$,
the image of $\gamma$ is of degree 1.
\item Under the composite morphism $\NE(\tY)\to\NE(Y_k)\xrightarrow{\sim}\NE(Y)$, the image of $\gamma$ equals $\hbeta\in\NE(Y)$.
\end{enumerate}
\end{defin}

We denote by $\tbeta\subset\NE(\tY)$ the finite set of elements in $\NE(\tY)$ representing $\hbeta$.

Let $\bcMon(\tY,\tbeta)$ denote the moduli stack of $n$-pointed rational stable maps into $\tY$ with class in $\tbeta$.
We will assume $n\ge 3$ in the following.
For $1\le i\le n$, we will denote by $\ev_i$ the evaluation morphism of the $i\th$ marked point.
We refer to \cite{Kontsevich_Enumeration_1995,Fulton_Stable_1997} for the classical theory of stable maps, to \cite{Yu_Gromov_2014} for the theory of stable maps in non-archimedean analytic geometry, and to \cite{Kontsevich_Gromov-Witten_1994,Behrend_Stacks_GW_1996,Li_Virtual_GW_algebraic_1998,Behrend_Intrinsic_normal_cone_1997,Behrend_Gromov-Witten_1997} for the theory of Gromov-Witten invariants in algebraic geometry.

\begin{lem}\label{lem:virtual_dim}
The virtual dimension of the moduli stack $\bcMon(\tY,\tbeta)$ equals $n+2$.
\end{lem}
\begin{proof}
For a smooth projective variety $V$ over a field, let $\bcMgn(V,\alpha)$ denote the moduli stack of $n$-pointed genus $g$ stable maps into $V$ with homology class $\alpha$.
The formula of virtual dimension (cf.\ \cite{Behrend_Gromov-Witten_1997,Li_Virtual_GW_algebraic_1998}) gives
\[(1-g)(\dim V -3) - \alpha\cdot K_V +n,\]
where $K_V$ denotes the canonical class of $V$.
Applying to our situation, we have $g=0$, $\dim\tY = 3$.
Moreover, for any $\gamma\in\tbeta$, we have $\gamma\cdot K_{\tY}=-2$ by \cref{def:class_beta}(\ref{item:class_beta:intersection}).
So the virtual dimension of $\bcMon(\tY,\tbeta)$ equals $n+2$.
\end{proof}

By \cite[Theorem 8.7]{Yu_Gromov_2014}, the moduli stack $\bcMon(\tY\an)$ of \kanal stable maps into $\tY\an$ is isomorphic to the analytification of the algebraic stack $\bcMon(\tY)$.

Let $\bcMon(\tY\an,\tbeta)$ denote the analytification of $\bcMon(\tY,\tbeta)$.
Let $\bcMon(\tY\an,\tbeta)_0$ denote the the substack consisting of stable maps such that the first marked point maps to $\tD\an_1$ and the second marked point maps to $\tD\an_2$.
Let $\bcMon(\tY\an,\tbeta,U_f)_0$ denote the substack of $\bcMon(\tY\an,\tbeta)_0$ consisting of \kanal stable maps whose composition with the retraction map $\ttau\colon\tY\an\to\tB'_f$ lies in $U_f$.

Let $W_1$ (resp.\ $W_2$) be a rational closed convex polyhedral neighborhood of $h(v_{10})$ (resp.\ $h(v_{20})$) satisfying the following conditions:
\begin{enumerate}[(i)]
\item $W_1\subset U(R)_f$, $W_2\subset U(R)_f$.
\item $W_1\cap\partial\tB'_f=\emptyset$, $W_2\cap\partial\tB'_f=\emptyset$.
\item $W_1 = \pi_1 (W_1)\times\pi_2 (W_1)$, $W_2 = \pi_1 (W_2) \times \pi_2(W_2)$. (Recall that $\pi_1\colon\tB\to\R$ and $\pi_2\colon\tB\to B$ denote the two projections.)
\item $O\notin \pi_2(W_1)$, $O\notin \pi_2(W_2)$, where $O$ is the origin of the tropical base $B$.
\end{enumerate}

By \cref{prop:extension_of_torus_fibration}, we can assume that the retraction map $\ttau\colon\tY\an\to\tB'_f$ is a trivial affinoid torus fibration over $W_1$ and $W_2$.
So $\ttau\inv(W_1)$ and $\ttau\inv(W_2)$ are affinoid domains inside $\tY\an$, which we denote by $\tW_1$ and $\tW_2$ respectively.
Let $A_1$ and $A_2$ be the $k$-affinoid algebras corresponding to $\tW_1$ and $\tW_2$ respectively.

\begin{lem}\label{lem:open_and_close}
Let $\cM$ be an analytic domain of $\bcMon(\tY\an,\tbeta)_0$.
Assume that the composition with the retraction map $\ttau\colon\tY\an\to\tB'_f$ of any \kanal stable map in $\cM$ hits the polyhedron $W_1$.
Let $\cM(U_f)\coloneqq\cM\cap\bcMon(\tY\an,\tbeta,U_f)_0$.
Then $\cM(U_f)$ is a union of connected components of $\cM$.
\end{lem}
\begin{proof}
Let $\cM(U(R)_f)$ (resp.\ $\cM(\oU(R)_f)$) denote the substack of $\cM(U_f)$ consisting of \kanal stable maps such that if we compose them with the retraction map $\ttau\colon\tY\an\to\tB'_f$, the images lie in $U(R)_f$ (resp.\ $\oU(R)_f$).

Let $I_{\tsX}$ denote the set of irreducible components of the special fiber $\tsX_s$.
For a subset $I\subset I_{\tsX}$, we denote by $\Delta^I$ the corresponding simplex in the Clemens polytope $\tB'_f$, by $(\Delta^I)^\circ$ its relative interior, and by $D^\circ_I$ the corresponding open stratum in $\tsX_s$.
Set
\[D(U(R)_f) \coloneqq \bigcup_{\substack{I\subset I_{\tsX}\\ (\Delta^I)^\circ \subset U(R)_f}} D_I^\circ,\qquad
D(\oU(R)_f) \coloneqq \bigcup_{\substack{I\subset I_{\tsX}\\ (\Delta^I)^\circ \subset \oU(R)_f}} D_I^\circ.\]
We note that $D(U(R)_f)$ is a closed subscheme of $\tsX_s$, and $D(\oU(R)_f)$ is an open subscheme of $\tsX_s$.

Let $\tfX$ denote the formal completion of $\tsX$ along the special fiber $\tsX_s$.
We note that $\tfX$ is a formal scheme proper and flat over the ring of integers $\kc$.
We have $\tfX_s \simeq \tsX_s$ and $\tfX_\eta \simeq \tY\an$ (cf.\ \cite{Berkovich_Vanishing_1994}).

Let $\bcMon(\tfX)$ denote the moduli stack of $n$-pointed rational formal stable maps into $\tfX$.
Let $\bcMon(\tfX_s)$ denote the moduli stack of $n$-pointed rational algebraic stable maps into $\tfX_s$.
By \cite[Proposition 8.6]{Yu_Gromov_2014}, we have
\[\bcMon(\tfX)_s\simeq \bcMon(\tfX_s) \simeq \bcMon(\tsX_s).\]
By \cite[Theorem 8.9]{Yu_Gromov_2014}, we have
\[\bcMon(\tfX)_\eta\simeq \bcMon(\tfX_\eta) \simeq \bcMon(\tY\an).\]
So we obtain a reduction map $\pi_\cM\colon\bcMon(\tY\an)\to\bcMon(\tsX_s)$.

Let $\bcMon(\tsX_s,D(U(R)_f))$ denote the closed substack of $\bcMon(\tsX_s)$ consisting of stable maps which map into $D(U(R)_f)$.

Let $\bcMon(\tsX_s,D(\oU(R)_f))$ denote the open substack of $\bcMon(\tsX_s)$ consisting of stable maps which map into $D(\oU(R)_f)$.

We have two isomorphisms
\begin{align*}
\pi\inv_\cM\big(\bcMon(\tsX_s,D(U(R)_f))\big)\cap\cM\simeq\cM(U(R)_f),\\
\pi\inv_\cM\big(\bcMon(\tsX_s,D(\oU(R)_f))\big)\cap\cM\simeq\cM(\oU(R)_f).
\end{align*}
Therefore, $\cM(U(R)_f)$ is an open substack of $\cM$, and $\cM(\oU(R)_f)$ is a closed substack of $\cM$.

The inclusions $U(R)_f\subset\oU(R)_f\subset U_f$ induce natural inclusions
\begin{equation}\label{eq:inclusions}
\cM(U(R)_f)\subset\cM(\oU(R)_f)\subset\cM(U_f).
\end{equation}
By the assumption that the composition with the retraction map $\ttau\colon\tY\an\to\tB'_f$ of any \kanal stable map in $\cM$ hits the polyhedron $W_1$, \cref{prop:rigidity_of_tropical_cylinder} implies that the inclusions in \eqref{eq:inclusions} induce natural isomorphisms
\[\cM(U(R)_f)\simeq\cM(\oU(R)_f)\simeq\cM(U_f).\]
Therefore, $\cM(U_f)$ is at the same time an open substack and a closed substack of $\cM$.
We conclude that $\cM(U_f)$ is a union of connected components of $\cM$.
\end{proof}

\begin{thm} \label{thm:properness_of_moduli}
Consider the two fiber products
\[\bcMon(\tY\an,\tbeta)_0\times_{\tY\an} \tW_1\quad \text{and}\quad\bcMon(\tY\an,\tbeta,U_f)_0\times_{\tY\an} \tW_1\]
given by the evaluation morphism $\ev_3$ of the third marked point.
Then the latter fiber product is a union of connected components of the former fiber product.
Therefore, the evaluation morphism $\ev_3\colon\bcMon(\tY\an,\tbeta,U_f)_0\to\tY\an$ is proper over $\tW_1$.
\end{thm}
\begin{proof}
Let $\cM\coloneqq\bcMon(\tY\an,\tbeta)_0\times_{\tY\an} \tW_1$ and $\cM(U_f)\coloneqq\cM\cap\bcMon(\tY\an,\tbeta,U_f)_0$.
We have $\cM(U_f)\simeq\bcMon(\tY\an,\tbeta,U_f)_0\times_{\tY\an} \tW_1$.
\cref{lem:open_and_close} implies that $\cM(U_f)$ is a union of connected components of $\cM$.
So we have proved the first statement of the theorem.

Since $\bcMon(\tY,\tbeta)$ is a proper algebraic stack, its analytification $\bcMon(\tY\an,\tbeta)$ is a proper $k$-analytic stack (cf.\ \cite[Proposition 6.4]{Porta_Yu_Higher_analytic_stacks_2014}).
So the closed substack $\bcMon(\tY\an,\tbeta)_0$ is also proper.
By base change, the \kanal stack $\cM$ is proper over $\tW_1$.
Therefore, the first statement of the theorem implies that the \kanal stack $\cM(U_f)$ is proper over $\tW_1$.
\end{proof}

\begin{cor}
The \kanal stack
\[\bcMon(\tY\an,\tbeta,U_f)_0 \times_{\tY\an} \tW_1\]
is isomorphic to the analytification of an algebraic stack proper over $\Spec A_1$, which we denote by $\bcMon(\tY\an,\tbeta,U_f,W_1)_0\alg$.
\end{cor}
\begin{proof}
Let $\cF$ denote the pushforward of the structure sheaf of the \kanal stack $\bcMon(\tY\an,\tbeta,U_f)_0 \times_{\tY\an} \tW_1$ to the \kanal stack $\bcMon(\tY\an,\tbeta)_0\times_{\tY\an} \tW_1$.
By \cref{thm:properness_of_moduli}, $\cF$ is a coherent sheaf on the \kanal stack $\bcMon(\tY\an,\tbeta)_0\times_{\tY\an} \tW_1$.
By construction, the \kanal stack $\bcMon(\tY\an,\tbeta)_0\times_{\tY\an} \tW_1$ is the analytification of an algebraic stack over $\Spec A_1$, which we denote by $\bcMon(\tY\an,\tbeta,W_1)_0\alg$.
Now we apply the GAGA theorem for \kanal stacks (\cite[Theorem 7.4]{Porta_Yu_Higher_analytic_stacks_2014}) to the algebraic stack $\bcMon(\tY\an,\tbeta,W_1)_0\alg$ over $\Spec A_1$.
We deduce that $\cF$ is isomorphic to the analytification of an algebraic coherent sheaf $\cF\alg$ on $\bcMon(\tY\an,\tbeta,W_1)_0\alg$.
The algebraic coherent sheaf $\cF\alg$ defines a closed substack of $\bcMon(\tY\an,\tbeta,W_1)_0\alg$, which we denote by $\bcMon(\tY\an,\tbeta,U_f,W_1)_0\alg$.
We note that $\bcMon(\tY\an,\tbeta,U_f,W_1)_0\alg$ is a union of connected components of $\bcMon(\tY\an,\tbeta,W_1)_0\alg$.
It is proper over $\Spec A_1$ because $\bcMon(\tY\an,\tbeta,W_1)_0\alg$ is proper over $\Spec A_1$.
By construction, the analytification of $\bcMon(\tY\an,\tbeta,U_f,W_1)_0\alg$ is the \kanal stack $\bcMon(\tY\an,\tbeta,U_f)_0 \times_{\tY\an} \tW_1$ we started with,
so we have proved the corollary.
\end{proof}

\bigskip
With the preparations above, we are ready to give the definition of the number $N(L,\beta)$.

Choose a rigid point $p$ of $\tW_1$.
We specialize to the case $n=3$.
Let
\[\bcM_{0,3}(\tY\an,\tbeta,U_f,p)_0\alg\]
denote the fiber of the morphism $\ev_3\colon\bcM_{0,3}(\tY\an,\tbeta,U_f,W_1)_0\alg\to\Spec A_1$ over $p$.
It is a proper algebraic stack over the residue field at $p$.
Let $V_p$ denote the restriction of the virtual fundamental class of $\bcM_{0,3}(\tY,\tbeta)_0$ to $\bcM_{0,3}(\tY\an,\tbeta,U_f,p)_0\alg$.
By \cref{lem:virtual_dim}, the cycle $V_p$ is of dimension zero.

\begin{defin}
We define the \emph{number of holomorphic cylinders} $N(L,\beta)$ to be the degree of the 0-dimensional cycle $V_p$.
\end{defin}

\begin{rem}
In the construction of the number $N(L,\beta)$, we used the affinoid domain $\tW_1\subset\tY\an$.
This raises a natural question: if we choose to use the affinoid domain $\tW_2\subset\tY\an$, do we obtain the same number?
It is not clear at all à priori.
We will give an affirmative answer to this question in \cref{sec:symmetry}.
\end{rem}

By construction, the number $N(L,\beta)$ does not depend on the choice of the point $p$.
Moreover, we have the following property of deformation invariance.

\begin{prop}[Deformation invariance]
Let $L_t,\, t\in[0,1]$ be a continuous deformation of extendable spines in the tropical base $B$.
Assume that the associated extended spines $\hL_t$ also deforms continuously.
Then we have $N(L_0,\beta)=N(L_1,\beta)$ for any $\beta\in\NE(Y)$.
\end{prop}
\begin{proof}
For all $t\in[0,1]$, let $Z_t$ denote the extended tropical cylinder in $\tB$ associated to the extended spine $\hL_t$,
and let $U_t$ denote the open neighborhood constructed in \cref{prop:rigidity_of_tropical_cylinder} for the extended tropical cylinder $Z_t$.
Now fix $t\in[0,1]$.
By continuity, there exists $\epsilon>0$, such that for any $t'\in (t-\epsilon,t+\epsilon)\cap[0,1]$, the extended tropical cylinder $Z_{t'}$ lies in $U_t$.
By construction, we have $N(L_t,\beta)=N(L_{t'},\beta)$ for all $t'\in (t-\epsilon,t+\epsilon)\cap[0,1]$.
Now the proposition follows from the compactness of the interval $[0,1]$.
\end{proof}

\begin{rem}\label{rem:curve_class}
	The curve class $\beta$ does not play a big role in this paper.
	Nevertheless, we remark that in our construction, $\beta$ represents exactly the class of our holomorphic cylinders in the sense of the following definition.
\end{rem}

\begin{defin}
	Let $C$ be a compact quasi-smooth strictly $k$-analytic curve and $f\colon C\to Y\an_{k}$ a morphism.
	Up to passing to a finite ground field extension, we can choose a strictly semistable formal model $\fC$ of $C$ over $\kc$ such that $f\colon C\to Y\an_{k}$ extends to a morphism $\ff\colon\fC\to\widehat Y_{\kc}$, where $\widehat Y_{\kc}$ denotes the formal completion of $Y_{\kc}$ along its special fiber.
	Let $\fC^p_s$ be the union of proper irreducible components of the special fiber $\fC_s$ of $\fC$.
	We define $[\ff_s(\fC^p_s)]\in\NE(Y)$ to be the class of $f$.
	Since two different choices of the model $\fC$ can always be dominated by another model, we see that the curve class does not depend on the choice.
	If the domain curve $C$ is nodal, we make the definition after taking normalization.
\end{defin}

\section{The symmetry property} \label{sec:symmetry}

In \cref{sec:hol_cylinders}, we have defined the number of holomorphic cylinders $N(L,\beta)$ given a spine $L=(\Gamma_0,(v_{10},v_{20}),h_0)$ and a class $\beta\in\NE(Y)$.

A natural question is whether the number $N(L,\beta)$ depends on the orientation of the spine $L$.
We give an affirmative answer in \cref{thm:symmetry}.
We start with two lemmas concerning Berkovich non-archimedean analytic spaces.

\begin{lem}\label{lem:maximum_modulus}
let $U$ be a $k$-affinoid space and $f$ a holomorphic function on $U$.
If the norm $|f|$ reaches maximum at a point $x_0$ in the interior $\Int(U)$ of $U$,
then there exists a neighborhood $V$ of $x_0$ in $U$ such that $|f|$ is constant on $V$.
\end{lem}
\begin{proof}
Let $D$ be the closed disc of radius $|f(x_0)|$.
Then $f$ defines a morphism $U\to D$.
By \cite[Proposition 2.5.8(iii)]{Berkovich_Spectral_1990}, we have the formula
\[ \Int(U) = \Int(U/D)\cap f\inv(\Int(D)).\]
Since $x_0\in\Int(U)$ by assumption, we deduce that $f(x_0)\in\Int(D)$.
Let $V_0$ be the connected component of $\Int(D)$ that contains $f(x_0)$.
Then $V\coloneqq f\inv(V_0)$ is a neighborhood of $x_0$ in $U$ such that $|f|$ is constant on $V$, completing the proof of the lemma.
\end{proof}

\begin{lem}\label{lem:function_of_norm_1}
Let $M$ be a connected \kanal space without boundary, $m_0\in M$ a rigid point and $f$ a holomorphic function on $M$.
Assume that the norm of $f$ is constantly one on $M$.
Then the norm of the difference $(f-f(m_0))$ is strictly less than one everywhere on $M$.
\end{lem}
\begin{proof}
Let
\[ M_1 = \big\{ m\in M \ \big|\ |f(m)-f(m_0)|<1 \big\}.\]
Since for any $m\in M$,
\[ | f(m)-f(m_0)|\leq\max\big\{ |f(m)|, |f(m_0)|\big\} = 1,\]
we have
\begin{equation}\label{eq:MminusM1}
M\setminus M_1 = \big\{ m\in M\ \big|\ |f(m)-f(m_0)| = 1\big\}.
\end{equation}

Assume by contradiction that $M\setminus M_1 \neq \emptyset$.
Let $m'\in M\setminus M_1$.
Since $M$ has no boundary, there exists an affinoid neighborhood $U$ of $m'$ in $M$ such that $m'$ is contained in the interior of $U$.
Consider the function $f(m)-f(m_0)$ depending on $m$.
We have $|f(m)-f(m_0)|\le 1$ for all $m\in U$ and it reaches maximum at the point $m'\in U$.
By \cref{lem:maximum_modulus}, there exists a neighborhood $V$ of $m'$ in $U$ such that $|f(m)-f(m_0)|=1$ for all $m\in V$.
In other words, we have $V\subset M\setminus M_1$.
Applying the argument above to every $m'\in M\setminus M_1$, we deduce that $M\setminus M_1$ is an open subset of $M$.

By \cref{eq:MminusM1}, $M\setminus M_1$ is also a closed subset of $M$.
Since $M$ is connected by assumption, we see that $M\setminus M_1=M$, in other words $M_1=\emptyset$.
This contradicts the fact that $m_0\in M$, completing the proof of the lemma.
\end{proof}

Let $L = (\Gamma_0,(v_{10},v_{20}),h_0\colon\Gamma_0\to B)$ be an extendable spine in the tropical base $B$.
Let $L\sw = (\Gamma_0,(v_{20},v_{10}),h_0\colon\Gamma_0\to B)$ be the same spine except that the vertices $v_{10}$ and $v_{20}$ are switched.
We have the following symmetry property of the counting of holomorphic cylinders.

\begin{thm}[Symmetry property]\label{thm:symmetry}
We have
\[ N(L,\beta) = N(L\sw,\beta),\]
for any class $\beta\in\NE(Y)$.
\end{thm}
\begin{proof}
We use the settings of \cref{sec:hol_cylinders}.
Let $\Psi$ be the map
\[\Psi\coloneqq (\ev_3,\ev_4)\colon\bcM_{0,4} (\tY\an,\tbeta,U_f)_0\longrightarrow\tY\an\times\tY\an,\]
where $\ev_3,\ev_4$ denote respectively the evaluation map of the third and the fourth marked point.

Let $W_1^\circ$ and $W_2^\circ$ denote the interior of $W_1$ and $W_2$ respectively.
Let $\tW_1^\circ\coloneqq\ttau\inv(W_1^\circ)\subset\tW_1$,
$\tW_2^\circ\coloneqq\ttau\inv(W_1^\circ)\subset\tW_2$.
Let
\begin{align*}
&\cM\coloneqq\bcM_{0,4} (\tY\an,\tbeta,U_f)_0\times_{\tY\an\times\tY\an} (\tW_1^\circ\times\tW_2^\circ),\\
&\cM(U_f)\coloneqq\cM\cap\bcM_{0,4}(\tY\an,\tbeta,U_f)_0.
\end{align*}
We have
\[\cM(U_f) \simeq \bcM_{0,4} (\tY\an,\tbeta,U_f)_0\times_{\tY\an\times\tY\an} (\tW_1^\circ\times\tW_2^\circ).\]

Recall that $\pi_1\colon\tB\to\R$ and $\pi_2\colon\tB\to B$ denote the two projections.
Let
\begin{align*}
\tW_1'&\coloneqq \tW_1^\circ \times\val\inv(\pi_1(W_2^\circ)),\\
\tW_2'&\coloneqq \val\inv(\pi_1(W_1^\circ))\times \tW_2^\circ.
\end{align*}
Let $p_1\colon \tW_1^\circ\to\val\inv(\pi_1(W_1^\circ))$ and $p_2\colon \tW_2^\circ\to\val\inv(\pi_1(W_2^\circ))$ denote the two projections.
Let
\begin{align*}
\Psi_1&\coloneqq (\id\times p_2)\circ\Psi\colon\cM(U_f)\to\tW_1',\\
\Psi_2&\coloneqq (p_1\times\id)\circ\Psi\colon \cM(U_f)\to\tW_2'.
\end{align*}

Up to shrinking the polyhedrons $W_1$ and $W_2$, we can embed $\tW'_1$ and $\tW'_2$ into $(\Gmk^4)\an$ so that
\begin{enumerate}[(i)]
\item there exists an open convex polyhedron $W\subset\R^4$ such that both $\tW'_1$ and $\tW'_2$ are isomorphic to the analytic domain $\tW\coloneqq(\val^4)\inv(W)\subset (\Gmk^4)\an$,
\item and that we have $\val^4\circ\Psi_1 = \val^4\circ\Psi_2$.
\end{enumerate}

Let $\cM_0$ be a connected component of $\cM(U_f)$.
Fix a rigid point $m_0\in\cM_0$.
Let $T$ denote the \kanal closed unit disc.
Consider the map
\begin{align*}
F\colon\cM_0\times T&\longrightarrow(\Gmk^4)\an\times T\\
(m,t)&\longmapsto\Big((1-t)\Psi_1(m) + t\cdot\frac{\Psi_1(m_0)}{\Psi_2(m_0)}\cdot\Psi_2(m),\ t\Big).
\end{align*}
Let $\widebar F$ be the composition of $F$ with the projection $(\Gmk^4)\an\times T\to(\Gmk^4)\an$.

\begin{lem}\label{lem:invariance_of_norm}
The following diagram commutes
\[
\begin{tikzcd}[column sep=6em]
\cM_0\times T\arrow{r}{\val^4\circ\,\widebar F} \arrow{d} &\R^4 \\
\cM_0\arrow{ru}[xshift=3.5em,yshift=-1em]{\val^4\circ\Psi_1}
\end{tikzcd}
\]
where the vertical arrow denotes the projection to the factor $\cM_0$ of the fiber product $\cM_0\times T$.
In other words, the composition $\val^4\circ\,\widebar F$ does not depend on $t\in T$.
\end{lem}
\begin{proof}
Let $g$ be the map
\begin{align*}
g\colon\cM_0&\longrightarrow(\Gmk^4)\an\\
m&\longmapsto\frac{\Psi_1(m)}{\Psi_2(m)}.
\end{align*}
We have $\val^4\circ\,g\equiv (0,\dots,0)\in\R^4$.

Let $G$ be the map
\begin{align*}
G\colon\cM_0\times T&\longrightarrow(\Gmk^4)\an\\
m&\longmapsto (1-t)g(m)+t g(m_0) = g(m) + t (g(m_0)-g(m)).
\end{align*}
By \cref{lem:function_of_norm_1}, we have $\val^4\circ\,(g(m_0)-g(m))\in\R_{>0}^4\subset\R^4$, for any $m\in\cM_0$.
Therefore, we have $\val\circ\,G\equiv(0,\dots,0)\in\R^4$ on $\cM_0\times T$.
Now the lemma follows from the equality $\widebar F=G\cdot\Psi_2$.
\end{proof}

Let $W'$ be a closed convex polyhedron in $W$.
Let $W''$ be a closed convex polyhedron in the interior of $W'$.
Let
\begin{align*}
&\tW'\coloneqq(\val^4)\inv(W')\subset(\Gmk^4)\an,\\
&\tW''\coloneqq(\val^4)\inv(W'')\subset(\Gmk^4)\an.\\
\end{align*}
Let
\begin{align*}
&\cM'_0\coloneqq(\val^4\circ\Psi_1)\inv(W'),\\
&\cM''_0\coloneqq(\val^4\circ\Psi_1)\inv(W'').
\end{align*}
By \cref{lem:invariance_of_norm}, the map $F$ induces two maps by restriction
\begin{align*}
&F'\colon\cM'_0\times T\to\tW'\times T,\\
&F''\colon\cM''_0\times T\to\tW''\times T.
\end{align*}

\begin{lem}\label{lem:properness_of_F}
The map $F''$ is a proper map.
\end{lem}
\begin{proof}
Since $\bcM_{0,4}(\tY,\tbeta)$ is a proper algebraic stack, its analytification $\bcM_{0,4}(\tY\an,\tbeta)$ is a proper $k$-analytic stack (cf.\ \cite[Proposition 6.4]{Porta_Yu_Higher_analytic_stacks_2014}).
So the closed substack $\bcM_{0,4}(\tY\an,\tbeta)_0$ is also proper.
By the definition of properness, there exists two finite affinoid quasi-smooth coverings $\{U_i\}_{i\in I}$ and $\{V_i\}_{i\in I}$ of $\bcM_{0,4}(\tY\an,\tbeta)_0$ such that $U_i\subset\Int(V_i)$ for every $i\in I$.

By \cref{lem:open_and_close}, $\cM(U_f)$ is a union of connected components of $\cM$.
Therefore, by base change, the coverings $\{U_i\}_{i\in I}$ and $\{V_i\}_{i\in I}$ of $\bcM_{0,4}(\tY\an,\tbeta)_0$ induce a finite affinoid quasi-smooth covering $\{U'_i\}_{i\in I'}$ of $\cM''_0\times T$ and a finite affinoid quasi-smooth covering $\{V'_i\}_{i\in I'}$ of $\cM'_0\times T$ such that $U'_i\subset\Int(V'_i/T)$.
Let $V''_i$ denote the restriction of $V'_i$ to $\cM''_0\times T$.
Then $\{V''_i\}_{i\in I}$ is a finite quasi-smooth covering of $\cM''_0\times T$ and we have $U'_i\subset\Int\big((V''_i/(\tW''\times T)\big)$.
So we have proved that $F''$ is a proper map.
\end{proof}

Let $w$ be a rigid point in $\tW''$.
\cref{lem:properness_of_F} implies that the degree of the virtual fundamental class on $(F'')\inv(w,0)$ equals the degree of the virtual fundamental class on $(F'')\inv(w,1)$.

Let $T_1$ denote the \kanal annulus $\val\inv(1)\subset\Gmk\an$.
Consider the map
\begin{align*}
H\colon\cM''_0\times T_1&\longrightarrow\tW''\times T_1\\
(m,t)&\longmapsto\big(t\cdot\Psi_2(m),t\big).
\end{align*}
The same proof of \cref{lem:properness_of_F} shows that $H$ is a proper map.
Therefore, the degree of the virtual line bundle on $H\inv(w,1)$ equals the degree of the virtual fundamental class on $H\inv\big(w,\frac{\Psi_1(m_0)}{\Psi_2(m_0)}\big)$.

Summing over every connected component $\cM_0$ of $\cM(U_f)$, we deduce that the degree of the virtual fundamental class on $\Psi_1\inv(w)$ equals the degree of the virtual fundamental class on $\Psi_2\inv(w)$.
By \cref{def:class_beta}(\ref{item:class_beta:degree_one}), using the divisor axiom of Gromov-Witten invariants (cf.\ \cite[\S 2.2.4]{Kontsevich_Gromov-Witten_1994}),
we observe that the degree of the virtual fundamental class on $\Psi_1\inv(w)$ equals the number $N(L,\beta)$, and the degree of the virtual fundamental class on $\Psi_2\inv(w)$ equals the number $N(L\sw,\beta)$.
So we have proved our theorem.
\end{proof}

\section{The example of a del Pezzo surface} \label{sec:example_del_Pezzo}

In this section, we compute the number of holomorphic cylinders for a concrete del Pezzo surface.
The result consists of certain binomial coefficients,
which verifies the Kontsevich-Soibelman wall-crossing formula for the focus-focus singularity (cf.\ \cref{rem:focus-focus}).

Let $Y_0 \coloneqq \bbP^1_\C \times \bbP^1_\C$.
Let $x$ be a coordinate on the first component of $\bbP^1_\C$ and let $y$ be a coordinate on the second component of $\bbP^1_\C$.
Let $D_{00}, D_{10}, D_{20}, D_{30}$ be the divisors on $Y_0$ given respectively by the equations $x=0, y=0, x=\infty, y=\infty$.
Let $Y$ be the blowup of $Y_0$ at the point with coordinate $(-1,0)$.
Then $Y$ is a del Pezzo surface of degree 7.
Let $D_0, D_1, D_2, D_3$ be the strict transforms of the divisors $D_{00}, D_{10}, D_{20}, D_{30}$ respectively.
Let $D=D_1+D_2+D_3+D_4$.
We see that $D$ is an anti-canonical cycle of rational curves in $Y$.
So $(Y,D)$ is a Looijenga pair in the sense of \cref{def:Looijenga pair}.

Using the notations of \cref{sec:log-CY_surfaces}, the tropical base $B$ associated to the pair $(Y,D)$ consists of four cones $\Delta_{0,1}, \Delta_{1,2}, \Delta_{2,3}, \Delta_{3,0}$.
Let $B\xrightarrow{\sim}\R^2$ be the homeomorphism which is linear on each of the four cones, and which maps the four cones subsequently to the first, second, third and fourth quadrant of $\R^2$.
We use this homeomorphism to identify $B$ with $\R^2$.
We remark that the \Zaffine structure on $\Delta_{i,i+1}$ is isomorphic to the restriction of the standard \Zaffine structure on $\R^2$ via the identification.
So we will describe the integral tangent vectors on $\Delta_{i,i+1}$ using the coordinates with respect to the standard basis of $\Z^2\subset\R^2$.

Let $l\in\Z_{>0}$ and $m,n\in\Z$.
Let $L(l,m,n)=(\Gamma,(v_1,v_2),h\colon\Gamma\to B\simeq\R^2)$ be a spine in the tropical base $B$ satisfying the following conditions.
\begin{enumerate}[(i)]
\item The \Zaffine tree $\Gamma$ has three vertices $v_0$, $v_1$, $v_2$ and two edges $e_1, e_2$.
The edge $e_1$ connects $v_0$ with $v_1$;
the edge $e_2$ connects $v_0$ with $v_2$.
\item We assume that $h(v_0)$ has coordinates $(0,b)$ with $b>0$.
\item We assume that $w_{v_0} (e_1) = (-l,-m+n)$ and $w_{v_0} (e_2) = (l,m)$.
\end{enumerate}

\begin{thm}\label{thm:cylinders_example}
The number of holomorphic cylinders associated to the spine $L(l,m,n)$ equals the binomial coefficient $\binom{l}{n}=\frac{l!}{n! (l-n)!}$.
In other words, we have
\[\sum_{\beta\in\NE(Y)} N(L(l,m,n),\beta)) = \binom{l}{n}.\]
\end{thm}
\begin{proof}
Since the surface $Y$ is the blowup of $Y_0$ at the point $(-1,0)$,
the interior $X\coloneqq Y\setminus D$ contains two piece of $(\C^\times)^2$.
Let $(x,u)$ be the coordinates on the first piece, with $y=u(x+1)$.
Let $(y,v)$ be the coordinates on the second piece, with $vy=x+1$.
The functions $x,y,v$ give an embedding of $X$ into $\C^\times\times\C\times\C$.
So $X$ is isomorphic to the surface in $\C^\times\times\C\times\C$ given by the equation $vy=x+1$.

Let $\widebar k$ be an algebraic closure of $k$ and let $X_{\widebar k}=X\times_{\Spec\C}\Spec\widebar k$.
Let $\bbG_{\mathrm m/\widebar k} = \Spec \widebar k [s,s\inv]$.
A morphism from $\bbG_{\mathrm m/\widebar k}$ to $X_{\widebar k}$ is given by three Laurent polynomials $\ox, \oy, \ov$ in $s$, such that $\ox$ is invertible and that the relation $\ov\oy=\ox+1$ holds.

Since $\ox$ is invertible, up to a change of coordinate of $s$, we can assume that $\ox = s^l$ for a nonnegative integer $l$.
Then we have $\ov\oy=1+s^l$.
Let us factorize
\[ 1+ s^l = (1+ b_1 s) ( 1+ b_2 s) \dots (1+b_l s),\]
where $b_1,\dots,b_l \in\widebar k$.
Let $J\subset\{1,\dots,l\}$ be a subset of cardinality $n$, and let $J^\mathrm{c}$ be its complement.
Let $c\in\widebar k$ with valuation greater than 0.
Let
\begin{align*}
\ox &= s^l,\\
\oy &= c\cdot s^m \prod_{j'\in J^{\mathrm c}} (1+b_{j'} s),\\
\ov &= c\inv\cdot s^{-m} \prod_{j\in J} (1+b_j s).
\end{align*} 

Let $f\colon\bbG_{\mathrm m/\widebar k}\to X_{\widebar k}$ be the morphism given by the functions $\ox,\oy,\ov$ above.
The composition of the retraction map $\tau\colon X_k\an\to B$ with the piecewise linear identification $B\xrightarrow{\sim}\R^2$ is given by
\[
\begin{cases}
\big(\val x, \; \min (0,\val y)\big) &\text{ when }\val x=0, \val v\ge 0,\\
\big(\val x, \; -\val v\big) &\text{ otherwise.}
\end{cases}
\]
Therefore, when $l>0$, the morphism $f$ gives rise to an extended spine in $B$ associated to a spine of the form $L(l,m,n)$.
Using the notations in \cref{sec:hol_cylinders}, the domain curves of the stable maps in the moduli stack $\bcM_{0,3} (\tY\an,\tbeta,U_f)_0\times_{\tY\an} \{p\}$ are all irreducible,
because the variety $X$ is affine in our example.
Moreover, since a projective line with 3 marked points has neither moduli parameter nor nontrivial automorphisms,
the stable maps in $\bcM_{0,3} (\tY\an,\tbeta,U_f)_0\times_{\tY\an} \{p\}$ correspond to morphisms $f\colon\bbG_{\mathrm m/\widebar k}\to X_{\widebar k}$ of the forms considered above.
The choice of the constant $c$ is uniquely determined by the choice of the rigid point $p\in\tW_1$.
So it remains a finite number of choices for the functions $\ox,\oy,\ov$, which correspond to the choice of the subset $J\subset \{1,\dots,l\}$ of cardinality $n$.

We conclude that
\[\coprod_{\beta\in\NE(Y)} \bcM_{0,3} (\tY\an,\tbeta,U_f)_0\times_{\tY\an} \{p\}\]
is a disjoint union of $\binom{l}{n}$ reduced points.
In other words, we have proved that
\[\sum_{\beta\in\NE(Y)} N(L,\beta) = \binom{l}{n}.\]
\end{proof}

\begin{rem} \label{rem:focus-focus}
The binomial coefficients in \cref{thm:cylinders_example} are related to the Kontsevich-Soibelman wall-crossing transformation around a focus-focus singularity.
We recall that the wall-crossing transformation around a two-dimensional focus-focus singularity is an automorphism $\varphi$ of the algebra $\C\llb x,y\rrb$ given by
\[\varphi(x) = x(1+y),\quad \varphi(y)=y.\]
The reference is \cite[\S 8]{Kontsevich_Affine_2006}, see also the generalizations by Gross and Siebert \cite{Gross_Real_Affine_2011,Gross_Invitation_2011}.
We compute that
\[\varphi(x^l y^m) = x^l (1+y)^l y^m = \sum_n \binom{l}{n}\, x^l y^{m+n}.\]
In other words, we obtain the following identity
\[\varphi(x^l y^m) = \sum_n \sum_{\beta\in\NE(Y)} N(L(l,m,n),\beta)\; x^l y^{m+n}.\]

Deeper relations between the enumeration of cylinders and the wall-crossing structures will be explored in a subsequent paper. 
\end{rem}